\documentclass{amsart}

\usepackage{amsmath,amssymb}
\usepackage{amsfonts}
\usepackage{graphicx}
\usepackage[all]{xy}
\usepackage{tikz}
\usepackage{pgfplots}
\pgfplotsset{compat=1.7}
\usepgfplotslibrary{polar}
\usetikzlibrary{decorations.pathreplacing,shapes.misc}
\usetikzlibrary{arrows,snakes,backgrounds}
\usepackage{caption}

\numberwithin{equation}{section}

\DeclareMathOperator{\Aut}{Aut}
\DeclareMathOperator{\Hom}{Hom}

\def\D{\mathcal{D}}
\def\sli{\mathcal{P}}
\def\Stab{\mathrm{Stab}}
\def\lto{\longrightarrow}
\def\iso{\xrightarrow{\sim}}
\def\xto{\xrightarrow}
\def\EG{\mathrm{EG}}

\newtheorem{thm}{Theorem}[section]
\newtheorem{defi}[thm]{Definition}
\newtheorem{prop}[thm]{Proposition}
\newtheorem{cor}[thm]{Corollary}

\newtheorem{lem}[thm]{Lemma}

\author{Chien-Hsun Wang}
\address{National Center for theoretical Sciences, Taipei, Taiwan}
\email{chwang@ncts.ntu.edu.tw}

\begin{document}

\title[Stability conditions and braid group actions on affine $A_n$ quivers]
{Stability conditions and braid group actions on affine $A_n$ quivers}

\begin{abstract}

We study stability conditions on the Calabi-Yau-$N$ categories associated to an affine type $A$ quiver which can be constructed from certain meromorphic quadratic differentials with zeroes of order $N-2$. We follow Ikeda's work to show that this moduli space of quadratic differentials is isomorphic to the space of stability conditions quotient by the spherical subgroup of the autoequivalence group. We show that the spherical subgroup is isomorphic to the braid group of affine type $A$ based on Khovanov-Seidel-Thomas method.

\end{abstract}

\maketitle

\section{Introduction}

Bridgeland stability conditions of quivers with potentials arose in the study of representation theory and physics.
Certain classes of ($10-2N$) dimension supersymmetric gauge theory can be engineered by compactification of ten dimensional superstring theories on Calabi Yau $N$-folds \cite{KKV}. The moduli space of gauge theories can be described by a lattice $\Gamma$ of charges and a linear function $Z:\Gamma\rightarrow\mathbb{C} $ called central charge. Supersymmety conditions require the mass $M$ of an object with charge $\gamma\in\Gamma$ to satisfy the inequality 
\begin{equation*}
	|Z(\gamma)|\leq M.
\end{equation*}
An object in the theory is called semistable (or BPS states in physics terminology) if its central charge saturates the above bound and each semistable object has a well defined phase $\theta=\arg Z$. Douglas introduced the concept of $\Pi$ stability condition \cite{DFR} that new semistable objects may appear when the phases of two semistable objects coincide as varying in the moduli space. When those happen, there is a wall of marginal stability for an object in the moduli space. Thus, we have the structure of wall-and-chamber for the moduli space. In a chamber on the one side of the wall, an object satisfying the stability condition is a semi-stable object. While in other chamber on the other side of the wall, the object does not satisfy the stability condition, so it would not be a semistable object and other new semistable objects may appear in that chamber.

The spectrum of semistable objects and their interactions can be encoded by $N-2$ colored quiver $Q$ with potentials $W$. Vector fields correspond to vertices of the quiver and matter fields correspond to arrows, while potential terms encode the interaction between matter fields \cite{ACC}. Lower dimension gauge theories contain more matter fields. This reflects that corresponding quivers have more possible degrees of arrows. Thus, the study of higher colored quivers would have interesting applications both in mathematics and physics \cite{FM}. 

When we move around in the moduli space, the central charge function and the set of semistable objects may change so the corresponding quiver would also be different. The operation of mutation of quiver $Q$ at a vertex gives the new quiver with potential $Q'$ when crossing the wall. The definition of mutation is given in section 4 below. We consider the derived category $D_{Q,W}$ of finite dimensional modules over a dg Ginzburg algebra of a quiver with potential. The result of Keller and Yang states that there is a pair of triangulated equivalence between $D_{Q,W}$
and the derived category $\mathcal{D}(Q',W')$ of a quiver with potential mutated at the vertex \cite{KY}. 

Seiberg and Witten \cite{SW} showed that the information of the data of 4-dimensional supersymmetric gauge theories can be encoded in a family of Riemann surfaces which vary over the moduli space and the central charge can be represented by the period integral of the square root of a meromorphic quadratic differential $\lambda=\sqrt{\phi}$ for a given charge $\gamma\in\Gamma$
\begin{equation*}
	Z(\gamma)=\int_{\gamma}\lambda.
\end{equation*}	
Trajectories of the associated quadratic differential on the marked surface with poles determine a triangulation of the surface. Gaiotto-Moore-Neitzke's work implies the relation between the mutation of quiver at a vertex and flip of a triangulation \cite{GMN}. This connects the wall-crossing of the quiver algebra and geometric description of certain classes of 4-dimensional gauge theories.

Bridgeland \cite{Br1} introduced the abstract definition of stability conditions $\sigma = (Z, \sli)$ on a triangulated category $\D$ which consists of a group  homomorphism from Grothendieck group of $\D$ to complex number $Z \colon K(\D) \to \mathbb{C}$ and a family of full additive subcategories $\sli (\theta) \subset \mathcal{D}$ of semistable objects of phase for $\theta \in \mathbb{R}$ satisfying certain axioms. The space of stability conditions $\mathrm{Stab}(\mathcal{D})$ is a smooth complex manifold, locally modelled on a vector space of central charge $\mathrm{Hom}(K(\D),\mathbb{C})$. It admits actions of the autoequivalence group $\mathrm{Aut}(\mathcal{D})$.

For $N=3$, Bridgeland and Smith \cite{BrSm} constructed a distinguished component $\mathrm{Stab}^{\circ}(\mathcal{D}_{Q,W})$ containing stability conditions whose heart is one of the standard hearts in the derived category $D_{Q,W}$ and prove that the orbifold of $\mathrm{Stab}^{\circ}(\mathcal{D})$ qoutient by the subgroup of autoequivalence $\mathrm{Aut}(\mathcal{D})$ which preserve the component is isomorphic to the moduli space of quadratic differentials.

This result is generalized to categories of dg Ginzburg algebras of quiver with potential which can be constructed from marked surfaces for any $N$ \cite{IQ}. In this paper, we will follow Ikeda's method to prove the case of affine type $A$ quiver for any $N$.

\subsection{Statement of results}
We consider a meromorphic quadratic differential
\begin{equation*}
	\phi(z)=\bigg[\frac{\prod_{i=1}^{n}(z-a_i)}{z^{p}}\bigg]^{N-2}\frac{dz^{\otimes 2}}{z^2}
\end{equation*}  on the Riemann surface $\mathbb{P}^1$
which has a pole of order $(N-2)p+2$ at $0$ and a pole of order $(N-2)(n-p)+2$ at $\infty$ and $n$ zeroes of order $N-2$. Denote by $\mathrm{Quad}(N,n)$ the moduli space of quadratic differentials with such prescribed singularities. We can define the homology groups $H_{+}(\phi)$ for $N$ even and $H_{-}(\phi)$ for $N$ odd.
Let $\Gamma$ be a free abelian group of rank $n$. The $\Gamma$-framing of $\phi\in \mathrm{Quad}(N,n)$ is an isomorphism of abelian groups $\psi: \Gamma\rightarrow H_{\pm}(\phi)$. Let $\mathrm{Quad}(N,n)^{\Gamma}$ be the moduli space of framed diffrentials and $\mathrm{Quad}(N,n)^{\Gamma}_*$ be the connected component containing a fixed framed differential.

Denote $D(\tilde{A}_{n,N})$ by the derived category of dg modules over the dg quiver algebra of affine type $A$. There is a distinguished component $\mathrm{Stab}^{\circ}(D(\tilde{A}_{n,N}))\subset\mathrm{Stab}(D(\tilde{A}_{n,N}))$ which contains stability conditions with the standard heart. Seidel and Thomas define the spherical twist functor in the $N$-spherical objects of the heart in the derived category \cite{ST}. Denote by $\mathrm{Sph}(\tilde{A}_{n,N})$ the spherical subgroup generated by spherical twist functors. Let $\mathrm{Sph}^{0}(\tilde{A}_{n,N})$ be the subgroup of spherical twists which act as the identity on $K(D(\tilde{A}_{n,N}))$. We prove that there is a $\mathbb{C}$-equivariant isomorphism 
\begin{equation*}
	\mathrm{Quad}(N,n)_{*}^{\Gamma} \simeq \mathrm{Stab}^{\circ}(D(\tilde{A}_{n,N})) / \mathrm{Sph}^{0}(\tilde{A}_{n,N}).
\end{equation*}
We follow arguments of of \cite{BrSm} and combine the following two results to prove this theorem. First, King and Qiu \cite{KQ} establishes an equivalence between the mutation of cluster category and tilted heart of the dervied category of dg algebras for acyclic quivers. Second, Torkildsen \cite{Tor12} constructs the geometric description of an affine type $A$ quiver and its relation to the corresponding cluster categories. He shows that the mutation class of $N$-angulation of the annulus with marked points quotient by some relations is bijective to the mutation class of the associated colored quiver of affine type $A$. 

Trajectories called $N-2$ diagonals of $\phi(z)$ connecting marked points on annulus give rise to an $N$-angulation $\Delta$. We associate the colored quiver $Q$ of affine type $A$ to the $N$-angulation $\Delta$. As varying the meromorphic differential $\phi$, an $N$-angulation would be mutated to a new $N$-angulation when crossing some walls in $\mathrm{Quad}(N,n)$. By using the results of \cite{KQ} and \cite{Tor12},  we obtain a bijection between $N-2$-diagonals and simple objects of a heart in $D(\tilde{A}_n)$. This implies that there is an isomorphism between the Grothendieck group of $D(\tilde{A}_n)$ and periods of the differential $\phi(z)$.	

Seidel and Thomas showed that the spherical subgroup of the derived category of of type $A$ quiver is isomorphic to the braid group of type $A$. Ikeda can apply this result directly to determine the group of autoequivalences of derived category of the type $A$ quiver algebras.

For the quiver of affine type $A$, we will apply Gadbled-Thiel-Wagner's work of the faithfulness of the categorical action. The main idea of their proof is to associate the geometric intersection number of curves on the disc and the Poincare polynomial of morphisms between objects of the category. Then we rework the procedure of Seidel-Thomas \cite[section 4]{ST} and show that the spherical subgroup $\mathrm{Sph}(\tilde{A}_{n,N})$ is isomorphic to the affine type $A$ braid group $B_{\tilde{A}_n}$.

The paper is organized as follows. In the second section, we review the quiver algebra of affine type $A$ and give the proof of the proposition mentioned in the last paragraph. In the third section, we define the quadratic differential $\phi$ on the Riemann surface which is a annulus and trajectories induced by $\phi$. In the fourth section, we show the bijection between the mutation of the N-angulation of the annulus and the mutation of the quiver of affine type $A$. In the fifth section, we review the definition of Bridgeland stability conditions and apply the framework of Bridgeland and Smith \cite[Theorem 1.2]{BrSm} to prove the main theorem.

\section{Derived Category of quiver algebras and Braid groups}
We denote $\mathbb{K}$ an algebraically closed field.
Let $Q=(Q_0,Q_1)$ be an acyclic quiver and $\mathbb{K}Q$ its path algebra. To categorify the path algebra, we take the Ginzburg dg algebra $(\mathbb{K}\Gamma_NQ,d)$ \cite{Ginz} for the quiver by adding an opposite arrow $a^*$ with degree $N-2$ to each of original arrow $a$ and a loop $t_i$ at each vertex $i$ with degree $N-1$ for each vertex
with a differential $d : \mathbb{K}\Gamma_NQ \to \mathbb{K}\Gamma_NQ$ of degree $-1$ defined by
\begin{itemize}
	\item $da =da^*= 0$ for $a \in Q_1$;
	\item $d t_i = e_i \,\left(\sum_{a \in Q_1}(aa^* -a^*a)\right)\,e_i$;
\end{itemize}
where $e_i$ is the idempotent at a vertex $i \in Q_0$. 

We take the derived category $D(k\Gamma_NQ)$ of dg modules over $\mathbb{K}\Gamma_NQ$, and consider the full sub-category $D_{fd}(\mathbb{K}\Gamma_NQ)$ of finite dimensional homology dg modules. Denote by $hom(E,F)$ the cochain complex of $\mathbb{K}$ vector space and $\mathrm{Hom}^i(E,F)=H^i(hom(E,F))$ its cohomology for objects $E$, $F$  of the category of cochain complexes of $\mathbb{K}\Gamma_NQ$-modules.

\begin{defi}
	An objects $S$ in $D_{fd}(\mathbb{K}\Gamma_NQ)$ is called $N$-spherical if
	\begin{align*}
		\Hom^i(S,S) = 
		\begin{cases}
			\mathbb{K} \quad \text{if} \quad i=0,N  \\
			\,0 \quad \text{otherwise} 
		\end{cases}
	\end{align*}
	and the composition
	\begin{equation*}
		\mathrm{Hom}^j(S,U)\times \mathrm{Hom}^{N-j}(U,S)\rightarrow Hom^N(S,S)\simeq \mathbb{K}
	\end{equation*}
	is a nondegenerate pairing for all $U\in D_{fd}(\mathbb{K}\Gamma_NQ),j\in\mathbb{Z}$.
\end{defi}

In \cite{Kel2} Keller shows that $D_{fd}(\mathbb{K}\Gamma_NQ)$ is a Calabi-Yau-$N$ category in the sense that for any two objects $E,F$, there is an isomorphism $\mathrm{Hom}(E,F)\simeq \mathrm{Hom}^N(F,E)^*$.

\begin{defi}
	For a spherical objects $S\in D_{fd}(\mathbb{K}\Gamma_NQ)$, we define the spherical twist functor $\Phi_{S}\in \mathrm{Aut}(D_{fd}(\mathbb{K}\Gamma_NQ))$ by
	\begin{equation*}
		\mathrm{Tw}_S(E)=\{\mathrm{Hom}^*(S,E)\otimes S\xrightarrow{\text{ev}} E\}
	\end{equation*}
	for any object $E\in D_{fd}(\mathbb{K}\Gamma_NQ)$.	  
\end{defi}
The map $ev$ is the evaluation map and $\mathrm{Tw}_S(E)$ is the cone of $ev$ defined by the cone of the complex
\begin{equation*}
	\mathrm{Hom}^*(S,E)\otimes S\xrightarrow{\text{ev}} E \rightarrow  \mathrm{Tw}_S(E).
\end{equation*}

We are interested in the case when $Q$ is the affine type $A$ quiver $\tilde{A}_n$ and we set its derived category $D(\tilde{A}_{n,N}):=D_{fd}(\mathbb{K}\Gamma_N\tilde{A}_n)$. The standard heart of $D(\tilde{A}_{n,N})$ is equivalent to dg modules over $\mathbb{K}\Gamma_N\tilde{A}_n$ \cite{Am} and
is generated by simple $\mathbb{K}\Gamma_N\tilde{A}_n$-modules $S_1,\cdots,S_n$ which are $N$-spherical objects in $D(\tilde{A}_{n,N})$ \cite{Kel2}. So we have spherical twists $\mathrm{Tw}_{S_1},\cdots ,\mathrm{Tw}_{S_n}\in\Aut(D(\tilde{A}_{n,N}))$.
We define the spherical twist subgroup $\mathrm{Sph}(\tilde{A}_{n,N})$ to be the subgroup of $\Aut(D(\tilde{A}_{n,N}))$ generated by spherical twists:
\begin{equation*}
	\mathrm{Sph}(\tilde{A}_{n,N})):=\langle \mathrm{Tw}_{S_1},\cdots ,\mathrm{Tw}_{S_n} \rangle.
\end{equation*}
\subsection{Braid groups of affine type $A$}
The extended affine type $A$ braid group $\hat{B}_{\tilde{A}_n}$ is generated by
$\sigma_i$ for $i=1,\cdots,n$ consisting of a crossing between the strands labelled $i$ and $i+1$ modulo $n$ and a cyclic permutation $\rho$ of points 1 to $n$. They satisfy the following relations:
\begin{align}
	\sigma_i\sigma_j&=\sigma_j\sigma_i\quad\quad\quad\quad  for\quad|i-j|\geq 2\\
	\sigma_i\sigma_{i+1}\sigma_i&=\sigma_{i+1}\sigma_i\sigma_{i+1} \quad\, for\quad i=1,\cdots,n \\
	\rho\sigma_i\rho^{-1}&=\sigma_{i+1}\quad\quad\quad\,\,\,\,\,\,for\quad i=1,\cdots,n
\end{align}
where the indices of the generators is modulo $n$, $\sigma_{n+i}=\sigma_i$. 

\begin{figure}
	\centering
	\begin{minipage}[b]{0.6\textwidth}	
		\begin{tikzpicture}[scale=3]
			(0,0) -- (3mm,0mm) arc (0:30:3mm) -- cycle;
			\draw (0,0) circle (1cm);
			\draw (0,0) circle (0.4cm);
			
			\node [label={[xshift=2cm, yshift=2.5cm]n}] {};
			\node [label={[xshift=0.5cm, yshift=0.4cm]n}] {};
			\node [label={[xshift=0cm, yshift=3cm]1}] {};
			\node [label={[xshift=0cm, yshift=0.45cm]1}] {};
			
			\node [label={[xshift=-2cm, yshift=2.5cm]2}] {};
			\node [label={[xshift=-0.5cm, yshift=0.36cm]2}] {};
			\node [label={[xshift=-3.2cm, yshift=0.9cm]3}] {};
			\node [label={[xshift=-0.85cm, yshift=0.0cm]3}] {};
			\node [label={[xshift=-3.3cm, yshift=-1.2cm]i}] {};
			\node [label={[xshift=-0.9cm, yshift=-0.6cm]i}] {};
			\node [label={[xshift=-1.4cm, yshift=-3.6cm]i+1}] {};
			\node [label={[xshift=-0.3cm, yshift=-1.2cm]i+1}] {};
			\node [label={[xshift=3.2cm, yshift=0.8cm]n-1}] {};
			\node [label={[xshift=0.7cm, yshift=0.0cm]n-1}] {};
			\node [label={[xshift=3.3cm, yshift=-1.5cm]n-2}] {};
			\node [label={[xshift=0.75cm, yshift=-0.7cm]n-2}] {};
			
			\node [label={[xshift=-3.28cm, yshift=0.3cm].}] {};
			\node [label={[xshift=-3.36cm, yshift=-0.1cm].}] {};
			\node [label={[xshift=-3.3cm, yshift=-0.5cm].}] {};
			\node [label={[xshift=2.68cm, yshift=-2.4cm].}] {};
			\node [label={[xshift=2.5cm, yshift=-2.64cm].}] {};
			\node [label={[xshift=2.3cm, yshift=-2.82cm].}] {};

			\fill (canvas cs:x=0.0cm,y=1.0cm) circle (1pt);
			\fill (canvas cs:x=0.0cm,y=0.4cm) circle (1pt);
			\fill (canvas cs:x=0.6cm,y=0.8cm) circle (1pt);
			\fill (canvas cs:x=0.24cm,y=0.32cm) circle (1pt);
			\fill (canvas cs:x=0.923cm,y=0.38cm) circle (1pt);
			\fill (canvas cs:x=0.368cm,y=0.15cm) circle (1pt);
			\fill (canvas cs:x=-0.6cm,y=0.8cm) circle (1pt);
			\fill (canvas cs:x=-0.24cm,y=0.32cm) circle (1pt);
			\fill (canvas cs:x=0.6cm,y=-0.8cm) circle (1pt);
			\fill (canvas cs:x=0.24cm,y=-0.32cm) circle (1pt);
			\fill (canvas cs:x=-0.923cm,y=0.38cm) circle (1pt);
			\fill (canvas cs:x=-0.369cm,y=0.15cm) circle (1pt);
			\fill (canvas cs:x=0.923cm,y=-0.38cm) circle (1pt);
			\fill (canvas cs:x=0.368cm,y=-0.15cm) circle (1pt);
			\fill (canvas cs:x=-0.96cm,y=-0.28cm) circle (1pt);
			\fill (canvas cs:x=-0.38cm,y=-0.923cm) circle (1pt);
			\fill (canvas cs:x=-0.384cm,y=-0.112cm) circle (1pt);
			\fill (canvas cs:x=-0.15cm,y=-0.369cm) circle (1pt);

			\draw  (0.6,0.8) -- (0.24,0.32) ;
			\draw  (0.923,0.38) -- (0.369,0.15) ;
			\draw  (0.923,-0.38) -- (0.368,-0.15) ;
			\draw  (-0.6,0.8) -- (-0.24,0.32) ;
			\draw  (0.6,-0.8) -- (0.24,-0.32) ;
			\draw  (-0.923,0.38) -- (-0.369,0.15) ;
			\draw  (0.0,1.0) -- (0,0.4) ;
			\draw  (-0.96,-0.28) .. controls (-0.5,-0.6) and (-0.15,-0.369) .. (-0.15,-0.369) ;
			\draw  (-0.38,-0.923).. controls (-0.46,-0.63) and (-0.45,-0.5) ..(-0.45,-0.5)  ;
			\draw  (-0.384,-0.112).. controls (-0.45,-0.3) ..(-0.45,-0.42)  ;

		\end{tikzpicture}
	\end{minipage}	
	\begin{minipage}[b]{0.6\textwidth}	
		\begin{tikzpicture}[scale=3]
			(0,0) -- (3mm,0mm) arc (0:30:3mm) -- cycle;
			\draw (0,0) circle (1cm);
			\draw (0,0) circle (0.4cm);
			
			\node [label={[xshift=2cm, yshift=2.5cm]n}] {};
			\node [label={[xshift=0.5cm, yshift=0.4cm]n}] {};
			\node [label={[xshift=0cm, yshift=3cm]1}] {};
			\node [label={[xshift=0cm, yshift=0.45cm]1}] {};
			
			\node [label={[xshift=-2cm, yshift=2.5cm]2}] {};
			\node [label={[xshift=-0.5cm, yshift=0.36cm]2}] {};
			\node [label={[xshift=-3.2cm, yshift=0.9cm]3}] {};
			\node [label={[xshift=-0.85cm, yshift=0.0cm]3}] {};
			\node [label={[xshift=-3.3cm, yshift=-1.2cm]i}] {};
			\node [label={[xshift=-0.9cm, yshift=-0.6cm]i}] {};
			\node [label={[xshift=-1.4cm, yshift=-3.6cm]i+1}] {};
			\node [label={[xshift=-0.3cm, yshift=-1.2cm]i+1}] {};
			\node [label={[xshift=3.2cm, yshift=0.8cm]n-1}] {};
			\node [label={[xshift=0.7cm, yshift=0.0cm]n-1}] {};
			\node [label={[xshift=3.3cm, yshift=-1.5cm]n-2}] {};
			\node [label={[xshift=0.75cm, yshift=-0.7cm]n-2}] {};
			\node [label={[xshift=1cm, yshift=-3.8cm]i+2}] {};

			\node [label={[xshift=-3.28cm, yshift=0.3cm].}] {};
			\node [label={[xshift=-3.36cm, yshift=-0.1cm].}] {};
			\node [label={[xshift=-3.3cm, yshift=-0.5cm].}] {};
			\node [label={[xshift=2.68cm, yshift=-2.4cm].}] {};
			\node [label={[xshift=2.5cm, yshift=-2.64cm].}] {};
			\node [label={[xshift=2.3cm, yshift=-2.85cm].}] {};

			\fill (canvas cs:x=0.0cm,y=1.0cm) circle (1pt);
			\fill (canvas cs:x=0.0cm,y=0.4cm) circle (1pt);
			\fill (canvas cs:x=0.6cm,y=0.8cm) circle (1pt);
			\fill (canvas cs:x=0.24cm,y=0.32cm) circle (1pt);
			\fill (canvas cs:x=0.923cm,y=0.38cm) circle (1pt);
			\fill (canvas cs:x=0.368cm,y=0.15cm) circle (1pt);
			\fill (canvas cs:x=-0.6cm,y=0.8cm) circle (1pt);
			\fill (canvas cs:x=-0.24cm,y=0.32cm) circle (1pt);
			\fill (canvas cs:x=0.28cm,y=-0.96cm) circle (1pt);
			\fill (canvas cs:x=0.24cm,y=-0.32cm) circle (1pt);
			\fill (canvas cs:x=-0.923cm,y=0.38cm) circle (1pt);
			\fill (canvas cs:x=-0.369cm,y=0.15cm) circle (1pt);
			\fill (canvas cs:x=0.923cm,y=-0.38cm) circle (1pt);
			\fill (canvas cs:x=0.368cm,y=-0.15cm) circle (1pt);
			\fill (canvas cs:x=-0.96cm,y=-0.28cm) circle (1pt);
			\fill (canvas cs:x=-0.38cm,y=-0.923cm) circle (1pt);
			\fill (canvas cs:x=-0.384cm,y=-0.112cm) circle (1pt);
			\fill (canvas cs:x=-0.15cm,y=-0.369cm) circle (1pt);

			\draw  (0.0,1.0) .. controls (0.3,0.65) .. (0.24,0.32) ;
			\draw  (-0.6,0.8) .. controls (-0.1,0.65) ..  (0.0,0.4) ;
			\draw  (0.6,0.8) .. controls (0.65,0.4) .. (0.369,0.15) ;
			\draw  (0.923,0.38) .. controls (0.72,-0.15) .. (0.369,-0.15) ;
			
			\draw  (-0.923,0.38) .. controls (-0.48,0.5) .. (-0.24,0.32) ;
			\draw  (-0.38,-0.923) .. controls (-0.5,-0.5) .. (-0.384,-0.112) ;
			\draw  (0.923,-0.38)  .. controls (0.5,-0.5) .. (0.24,-0.32) ;
			\draw  (0.28,-0.96) .. controls (0,-0.8) .. (-0.15,-0.369) ;

		\end{tikzpicture}
		
	\end{minipage}
	\centering
	\caption{ Top: the affine braid generators $\sigma_i$, Bottom: the cyclic permutation $\rho$,}
\end{figure}
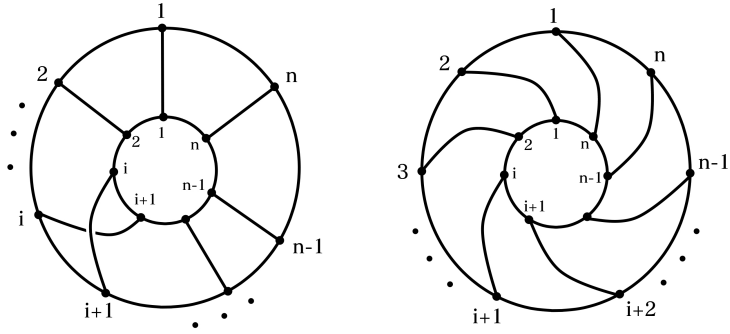

The central element of $\hat{B}_{\tilde{A}_n}$ is infinitely cyclic generated by $\rho^n$. We denote the subgroup, the braid group $B_{\tilde{A}_n}$ of affine type $A$ generated by $\sigma_1,\cdots,\sigma_n$ satisfying relations (2.1) and (2.2).
\subsection{Bigraded algebras of the affine type $A$ quiver}
To prove the faithful categorical action on the derived category of affine type $A$ quiver algebras , we introduce a differential graded algebra following Seidel-Thomas \cite{ST}. 
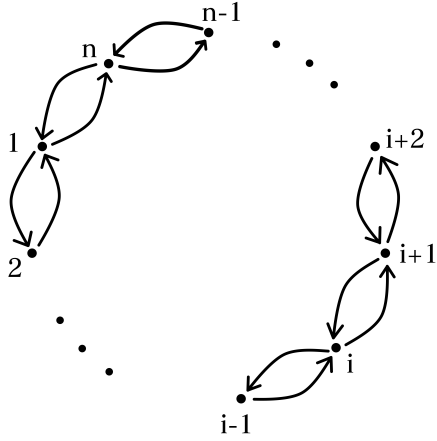
\begin{figure}
	
	\tikzstyle{pre}=[<-,shorten <=1pt,>=stealth',thick]
	\tikzstyle{post}=[<-,shorten <=1pt,>=stealth',thick]
	\tikzstyle{dot}=[circle,draw=black!100,fill=black!100,thick,
	inner sep=0pt,minimum size=2mm]
	\tikzstyle{sdot}=[circle,draw=black!100,fill=black!100,thick,
	inner sep=0pt,minimum size=0.8mm]

	\begin{center}	
		\begin{tikzpicture}[scale=2.5]
			\node (1) at (-0.923,0.38) [dot,draw,label=150:$1$] {};
			\node (2) at (-0.96,-0.28) [dot,draw,label=183:$2$] {};
			\node (i+2) at ( 0.96,0.28) [dot] {};
			\node (i-1) at ( 0,-1) [dot] {};
			\node (i) at (0.6,-0.8) [dot] {};
			\node (i+1) at (0.923,-0.38) [dot] {};
			\node[dot,draw,label=100:$n$] (n) at (-0.6,0.8)  {};
			\node[dot,draw,label=90:$n-1$] (n-1) at (0,1)  {}
			edge[pre,<-,bend right=45](n);
			\node[dot] (n) at (-0.6,0.8)  {}
			edge[post,<-,bend right=45](n-1)
			edge[post,<-,bend right=45](1);
			\node (1) at (-0.923,0.38) [dot] {}
			edge[post,<-,bend right=45](n)	
			edge[post,<-,bend right=45](2);	
			\node (2) at (-0.96,-0.28) [dot] {} 
			edge[post,<-,bend right=45](1);	
			\node (i+2) at ( 0.96,0.28) [dot,draw,label=0:$i+2$] {}
			edge[post,<-,bend right=45](i+1);	
			\node (i+1) at (0.923,-0.38) [dot,draw,label=315:$i+1$] {}
			edge[post,<-,bend right=45](i+2)
			edge[post,<-,bend right=45](i);
			\node (i) at (0.6,-0.8) [dot,draw,label=300:$i$] {}
			edge[post,<-,bend right=45](i+1)
			edge[post,<-,bend right=45](i-1);
			\node (i-1) at ( 0,-1) [dot,draw,label=268:$i-1$] {}
			edge[post,<-,bend right=45](i);
			
			\node (3) at (-0.38,-0.923) [sdot] {};
			\node (4) at (-0.8,-0.6) [sdot] {};
			\node (5) at (-0.6,-0.8) [sdot] {};
			\node (6) at (0.38,0.923) [sdot] {};
			\node (7) at (0.8,0.6) [sdot] {};
			\node (8) at (0.6,0.8) [sdot] {};
		\end{tikzpicture}	
	\end{center}
	\caption{The affine type zigzag quiver}	
\end{figure}

Consider the cyclic quiver $\Gamma_n$ associated with the affine type $A$ quiver which has the form in figure 2.
Denote by $(i_1|\cdots|i_m)$ the path starting at the vertex $(i_1)$ and ending at the vertex $(i_m)$ and $(i)$ be the path of length zero at vertex $i$.
Take $\mathbb{K}^{\oplus n}$ with generators $(1),\cdots,(n)$ as the ground ring. 
Denote by $R_n:=R_{n,N}$ the differential graded of the path algebra of $\Gamma_n$ quotient by the relations
$(i|i-1|i)=(i|i+1|i)$, $(i-1|i|i+1)=(i+1|i|i-1)=0$ for $n=1,\cdots,n$ modulo $n$. 
We assign two gradings to the path of $R_{n}$. 

\begin{itemize}
	\item The first grading is defined as $\mathrm{deg}(i)=0$, $\mathrm{deg}(i|i+1)=d_i$, $\mathrm{deg}(i+1|i)=N-d_i$, where $d_i=\frac{1}{2}N$ if $N$ is even and $d_i=\frac{1}{2}(N+(-1)^i)$ if $N$ is odd.
	\item The second grading is $\mathrm{deg}(1|n)=1$, $\mathrm{deg}(n|1)=-1$ while the degree of other elements is zero. We denote by $\langle-\rangle$ a shift in this grading.
\end{itemize}
We denote by $_{i}P=R_{n}(i)$ the indecomposable left projective module and by $P_i=(i)R_{n}$ the indecomposable right projective module and by $D(\operatorname{\mathit{R}_n-mod})$ the dervied category of differential graded modules over $R_n$

Let $S_1,\cdots, S_n$ be simple $\mathbb{K}\Gamma_N\tilde{A}_n$ modules which are $N$-spherical objects in $D(\tilde{A}_{n,N})$ and $S=\oplus_{i=1}^{n}S_i$ be their direct sum.
The chain complex of endomorphisms
\begin{equation*}
	end(S):=hom(S,S)=\bigoplus_{1\leq i,j\leq n} hom(S_i,S_j)
\end{equation*}
is a differential graded algebra. It has identity $id_{S_i}$ and the composition of homomorphisms gives the structure of multiplication of $end(S)$. The left multiplication with $id_{S_i}$ projects to $hom(S,S_i)$ and right multiplication with $id_{S_i}$ is $hom(S_i,S)$. The functor $hom(S,-):D(\tilde{A}_{n,N})\rightarrow K(end(S))$ maps an object $S_i$ to the dgm $hom(S,S_i)$ which is the indecomposable right projective module $P_i$. 

The graded algebra $R_n$ is related to the quiver algebra $\mathbb{K}\Gamma_N\tilde{A}_n$ by the following.  

\begin{lem}\label{coho}
	Suppose that the one-dimensional space $\mathrm{Hom}^{*}(S_{i+1},S_i)$ is concentrated in degree $d_i$ for each $i=1,\cdots,n$, then the cohomology algebra of the dga $end(S)$ is isomorphic to $R_n$.
\end{lem}
\begin{proof}
	Each $S_i$ are $N$-spherical objects, we have the non-degenerate pairings:
	\begin{align*}
		&\mathrm{Hom}^*(S_i,S_{i+1})\otimes \mathrm{Hom}^*(S_{i+1},S_i)\simeq \mathbb{K} \\
		&\mathrm{Hom}^*(S_{i+1},S_{i})\otimes \mathrm{Hom}^*(S_{i+1},S_i)\simeq \mathbb{K}
	\end{align*} 
	for $i=1,\cdots,n$. In particular, $\mathrm{Hom}^*(S_n,S_1)\otimes \mathrm{Hom}^*(S_1,S_n)\simeq \mathbb{K}$.
	The pairings are non-degenerate so $\mathrm{Hom}^*(S_i,S_{i+1})\simeq \mathbb{K}$ is concentrated in degree $N-d_i$. Without loss of generality, we have the equation
	$\alpha_i\beta_i=\beta_{i-1}\alpha_{i-1}$ for $\alpha_i\in \mathrm{Hom}^*(S_{i+1},S_i)$, $\beta_i\in \mathrm{Hom}^*(S_i,S_{i+1})$. Because $\mathrm{Hom}^*(S_i,S_j)=0$ for $|i-j|\geq 2$ we also have $\beta_i\beta_{i-1}=0$ and $\alpha_{i-1}\alpha_i=0$ for all $i=1,\cdots,n$.
	
	So there exists an isomorphism of $R_n\rightarrow \mathrm{Hom}^*(S,S)$ which maps $(i)$ to $id_{S_i}$, $(i|i+1)$ to $\alpha_i$ and $(i+1|i)$ to $\beta_i$.
\end{proof}
We define two complexes of $R_n$ for all $i=1,\cdots,n$:
\begin{align*}
	&T_i:0\rightarrow{}_i\mathit{P}\otimes {}P_i\xrightarrow{d_i}R_n\rightarrow 0\\
	&T'_i:0\xrightarrow{d'_i}R_n\rightarrow {}_i\mathit{P}\otimes {}P_i\{-1\}\xrightarrow{d_i}0
\end{align*}
with $R_{n}$ in degree $0$ and diffrentials
\begin{align*}
	d_i((i)\otimes(i))&=(i)\\
	d'_i((i)&=(i-1|i)\times(i|i-1)++(i+1|i)\otimes(i|i+1)\\&+(i)\otimes(i|i-1|i)+(i|i-1|i)\otimes(i)
\end{align*}
We consider the functors $\mathcal{T}_i=T_i\otimes_{R_n}-$ and $\mathcal{T'}_i=T'_i\otimes_{R_n}-$ which are are mutually inverse equivalence of categories \cite[Lemma 4.11]{KhSe}.

In the following, we construct a functor corresponding to the $\frac{1}{n}$ twist $\rho$ and a functor corresponding to the inverse operaton $\rho^{-1}$. We introduce an additional automorphism of ring $R_{n}$ as following. Define a bigraded $R_{n}$-module $\mathcal{R}_n^{\rho}$ to be $_1P\langle-1\rangle\oplus {}_2P\oplus \cdots\oplus _nP$.
Let $t_{\rho}$ be the automorphism of the ring $R_{n}$ sending a path $(i_1|i_2|\cdots|i_k)$ to $(i_1+1|i_2+1|\cdots|i_k+1)$. 

The left action of $R_{n}$ on $\mathcal{R}_n^{\rho}$ is the usual multiplication $a\cdot m=am$ for $a\in R_{n}, m\in \mathcal{R}_n^{\rho}$. While the right action of $R_{n}$ on $\mathcal{R}_n^{\rho}$ is the multiplication twisted by $t_{\rho}$,
$q\cdot a=qt_{\rho}(a)$ for $a\in R_{n}, q\in \mathcal{R}_n^{\rho}$. To distinguish the usual tensor product, we denote twisted tensor product $\hat{\otimes}$ for the bimodule $\mathcal{R}_n^{\rho}$ tensoring with the dgm with respect to the twisted multiplication defined above.

We also consider a bigraded $R_{n}$-module $^{\rho}\mathcal{R}_n$ to be $P_{1}\langle1\rangle\oplus P_{2}\oplus \cdots\oplus P_{n}$. 
The right action of $R_{n}$ on $^{\rho}\mathcal{R}_n$ is the usual multiplication while the left action of $R_{n}$ on $T_{\rho}$ is the multiplication twisted by $t_{\rho}$.

Consider two complexes of $R_{n}$-bimodule of length zero concentrated in cohomological degree zero:
\begin{align*}
	&\Omega_\rho:0\rightarrow \mathcal{R}_n^{\rho}\rightarrow 0.\\
	&\Omega'_\rho:0\rightarrow \mathcal{^{\rho}R}_n \rightarrow 0.
\end{align*}
Then we define the corresponding functors $\mathcal{T}_{\rho}=\Omega_\rho\hat{\otimes}_{R_{n}}-$ and $\mathcal{T}'_{\rho}=\Omega'_{\rho}\hat{\otimes}_{R_{n}}-$.

\begin{prop} 
	The functors $\mathcal{T}_{i}$ and $\mathcal{T}_{\rho}$ on $D(\operatorname{\mathit{R}_n-mod})$ are exact equivalences and satisfy the extended affine braid group relations (up to graded natural isomorphism).
	\begin{align*}
		\mathcal{T}_{i}\mathcal{T}_{i+1}\mathcal{T}_{i}&\simeq \mathcal{T}_{i+1}\mathcal{T}_{i}\mathcal{T}_{i+1}\quad\,\,\text{for}\quad i=1,\cdots,n \\
		\mathcal{T}_{i}\mathcal{T}_{j}&\simeq \mathcal{T}_{j}\mathcal{T}_{i}\quad\quad\quad\quad  \text{for}\quad|i-j|\geq 2 \\
		\mathcal{T}_{\rho}\mathcal{T}_{i}\mathcal{T}'_{\rho}&\simeq \mathcal{T}_{i+1}\quad\quad\quad\quad\text{for}\quad i=1,\cdots,n .
	\end{align*}
\end{prop}

\begin{proof}
	
	The first and second isomorphisms are the result of theorem 4.12 in \cite{ST} in the finite type $A$ case. The third isomorphism which only exists in the affine type $A$ case can be deduced by the properties of functors $\mathcal{T}_{\rho}$ and $\mathcal{T}'_{\rho}$. By direct computations,	we have the following isomorphisms of trigraded left $R_{n}$-modules
	\begin{equation*}
		\mathcal{R}_n^{\rho}\hat{\otimes}_{R_{n}} {}_iP\simeq {}_{i+1}P 
	\end{equation*}
	for all $i=1,\cdots,n-1$, and 
	\begin{equation}\label{shift}
		\mathcal{R}_n^{\rho}\hat{\otimes}_{R_{n}}{}_{n}P\simeq {}_1P\langle-1\rangle.
	\end{equation}
	and of trigarded right $R_{n}$-modules
	
	\begin{equation*}
		_i P\hat{\otimes}_{R_{n}}\mathcal{^{\rho}R}_n\simeq {}_{i+1}P 
	\end{equation*}
	for all $i=1,\cdots,n-1$, and 
	\begin{equation*}
		_nP\hat{\otimes}_{R_{n}}\mathcal{^{\rho}R}_n\simeq {}_1P\langle 1\rangle.
	\end{equation*}
	Combining these, we can see that 
	\begin{equation*}
		\mathcal{R}_n^{\rho}\hat{\otimes}_{R_{n}}{}_i P\otimes P_{i}\hat{\otimes}_{R_{n}}\mathcal{^{\rho}R}_n	\simeq {}_{i+1}P\otimes P_{i+1}
	\end{equation*}
	for all $i=1,\cdots,n$. This impies the third isomorphism.
\end{proof}

We consider the weak braid group action $\Phi_{n,N}: \hat{B}_{\tilde{A}_n}\rightarrow \mathrm{Auteq}(D(\operatorname{\mathit{R}_n-mod}))$ by sending $\sigma_1,\cdots\sigma_n,\rho$ to $\mathcal{T}_{1},\cdots,\mathcal{T}_{n},\mathcal{T}_{\rho}$ which is a group homomorphism by the above proposition.

\subsection{Geometric intersection numbers}

The method to prove the injectivity of categorical action is to relate the topology of curves on the surfaces. We recall some results and fix notations. Let $\mathbb{D}^*$ be a disc with a punctured point at the center which is equivalent to
an annulus contracting the interior boundary to a point and $\Delta$ be a set of $n$ distinct points. Denote by $\mathrm{MCG}(\mathbb{D}^*;\Delta)$ the mapping class group of oriented preserving homeomorphisms fixing the boundary and n+1 points.
A curve $c$ is the image of a smooth embedding of an interval into $\mathbb{D}^*$ such that the starting point and the end point both lie at points of $\Delta$.
For any curves $c_0, c_1$ which have minimal intersection number and do not form an empty bigon, we can define a geometric intersection number:
\begin{equation*}
	I(c_0,c_1)=|(c_0\cap c_1)\backslash\Delta|+\frac{1}{2}|(c_0\cap c_1)\cap \Delta|.
\end{equation*}
If $c_0$ and $c_1$ do not have minimal intersection, we can replace one of them to a isotopic curve such that they have minimal intersection. Consider a collection of curves $b_1,\cdots, b_n$ as in figure 3. We quote an important result from \cite[Lemma 4.2]{GTW}.

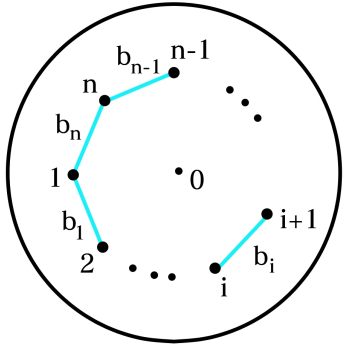
\begin{figure}
	\tikzstyle{pt}=[circle,draw=black!100,fill=black!100,thick,
	inner sep=0.1pt,minimum size=1.8mm]
	\tikzstyle{ssdot}=[circle,draw=black!100,fill=black!100,thick,
	inner sep=0pt,minimum size=0.5mm]
	\begin{center}	
		\begin{tikzpicture}[scale=2.5]
			(0,0) -- (3mm,0mm) arc (0:30:3mm) -- cycle;
			\draw (0,0) circle (1cm);	
			\node (0)[pt,draw,label=100:$0$] at (0,0) [pt] {};
			\node (n)[pt,draw,label=100:$n$] at (-0.36,0.48) [pt] {};
			\node[pt,draw,label=90:$n-1$]  (n-1) at (0.0,0.6) [pt] {};
			\node[pt,draw,label=180:$1$] (1) at (-0.5538,0.228) [pt] {};
			\node[pt,draw,label=180:$2$] (2) at (-0.576,-0.168) [pt] {};
			\node[pt,draw,label=360:$i+1$] (i+1) at (0.5538,-0.228) [pt] {};
			\node[pt,draw,label=270:$i$]  (i) at (0.36,-0.48) [pt] {};
			\draw[very thick,orange] (n-1)--(n);	
			\draw[very thick,orange] (1)--(n);
			\draw[very thick,orange] (1)--(2);
			\draw[very thick,orange] (i+1)--(i);
			\node[ssdot] (3) at (-0.228,-0.5538)  {};
			\node[ssdot] (4) at (-0.48,-0.36) {};
			\node[ssdot] (5) at (-0.36,-0.48)  {};
			\node[ssdot] (6) at (0.228,0.5538)  {};
			\node[ssdot] (7) at (0.48,0.36)  {};
			\node[ssdot] (8) at (0.36,0.48)  {};
		\end{tikzpicture}			
	\end{center}	
	\captionof{figure}{The affine configuration of the n+1-pinctured disk with curves connecting punctured points}
\end{figure}

\begin{lem}\label{mapping}
	If a mapping class $f\in \mathrm{MCG}(\mathbb{D}^*,\Delta)$ satisfying $I(b_j,f(b_i))=I(b_j,b_i)$ for all $i,j=1,\cdots,n$, then $f$ is in the center of $\mathrm{MCG}(\mathbb{D}^*,\Delta)$,  $f=\rho^{n\nu}$ for some $\nu\in \mathbb{Z}$.
\end{lem}
\subsection{Trigraded zig-zag algebras}

We recall the definition of the trigraded zig-zag algebra for the affine type $A$ quiver defined in \cite{GTW}. The cyclic quiver $\overline{\Gamma}_n$ of the affine type $A$ is of the form in figure 2.
We denote the path of length $m$ in $\overline{\Gamma}_n$ by $(\overline{i_0|\cdots|i_m})$ for $m$ tuples of numbers $i_1,\cdots,i_m\in \{1,\cdots,n\}$. The quotient ring $\overline{R}_n$ is a finite dimensional trigraded algebra of the path algebra $\mathbb{K}[\overline{\Gamma}_n]$ quotient by the relations: 
\begin{align*}
	(\overline{i|i+1|i})&=(\overline{i|i-1|i})\quad \text{for}\,\,i=1,\cdots,n \\
	(\overline{i-1|i|i+1})&=(\overline{i+1|i|i-1})=0\quad \text{for}\,\,i=1,\cdots,n
\end{align*}
The algebra $\overline{R}_n$ is similar to $R_n$ but with different gradings.

\begin{itemize}
	
	\item  The first grading is set to be $\mathrm{deg}(\overline{i|i+1})=\mathrm{deg}(\overline{i})=0$ and $\mathrm{deg}(\overline{i|i-1})=1$ for all $i$.
	\item  The second grading is the path length grading.
	\item The third grading is $\mathrm{deg}'(\overline{1|n})=1$, $\mathrm{deg}'(\overline{n|1})=-1$ while the degree of other elements is zero.
	
\end{itemize}
We denote by $\{-\}$ a shift in the first grading, by $(-)$ a shift in the second and by $\langle-\rangle$ a shift in the third.

Let $\operatorname{\overline{\mathit{R}}_n-mod}$ be the abelian category of finitely-generated graded right module over $\operatorname{\overline{\mathit{R}}_n}$
and $D^{b}(\operatorname{\overline{\mathit{R}}_n-mod})$ be its bounded derived category. 
Define a trigraded $R_n$-module $\overline{T}_{\rho}$ to be $\overline{P}_1\langle-1\rangle\oplus \overline{P}_2\oplus \cdots\oplus \overline{P}_n$. The automorphism $\{1\}$ shifting the grading of a module up by one  descends to an automorphism of $D^{b}(\operatorname{\overline{\mathit{R}}_n-mod})$.
For any objects $X,Y\in D^{b}(\operatorname{\overline{\mathit{R}}_n-mod})$ there is a trigraded vector space 
\begin{equation*}
	\bigoplus_{r_1,r_2,r_3} \mathrm{Hom}_{D^{b}(\operatorname{\overline{\mathit{R}}_n-mod})}(X,Y\{r_1\}[r_2]\langle r_3\rangle),
\end{equation*}
where $[-]$ is a shift in the cohomology grading.
Let $\overline{\mathit{P}}_i=(\overline{i})\overline{R}_n$ be the indecomposable right projective modules which are generated by the paths with left most vertex be $i$ for $1\leq i \leq n$. 
Let $\operatorname{\overline{\mathit{R}}_n-proj}$ be the full subcategories whose objects are direct sum of $\overline{\mathit{P}}_i\{r\}$ for $1,\cdots,n$ and $r\in\mathbb{Z}$.

Let $\overline{t}_{\rho}$ be the automorphism of the ring $\overline{R}_{n}$ sending a path $(\overline{i_1|i_2|\cdots|i_k})$ to $(\overline{i_1+1|i_2+1|\cdots|i_k+1})$. The left action of $\overline{R}_{n}$ on $\overline{T}_{\rho}$ is the ususal multiplication $r\cdot a=ra$ for $r\in \overline{R}_{n}, a\in \overline{T}_{\rho}$. While the right action of $\overline{R}_n$ on $T_{\rho}$ is the multiplication twisted by $\overline{t}_{\rho}$, defined by $b\cdot r=b\overline{t}_{\rho}(r)$ for $r\in \overline{R}_n, b\in \overline{T}_{\rho}$.

Define exact functors $\mathcal{F}_1,\cdots,\mathcal{F}_n,\mathcal{F}_{\rho}$ from $D^{b}(\operatorname{\overline{\mathit{R}}_n-mod})$ to itself by
\begin{equation*}
	\mathcal{F}_{i}(\mathit{X})=\big\{\mathit{X}\otimes_{\overline{R}_n}{}_i\overline{\mathit{P}}\otimes_{k}\overline{\mathit{P}}_i\longrightarrow X\big\}\quad\text{for $\mathit{i}$\,=\,1$\cdots$ n},
\end{equation*}	
and
\begin{equation*}
	\mathcal{F}_{\rho}(X)=\big\{T_{\rho}\otimes_{\overline{R}_n} X\big\}.
\end{equation*}
Those functors $\mathcal{F}_1,\cdots,\mathcal{F}_n,\mathcal{F}_{\rho}$ and their inverse functor are exact equivalences and generate a weak braid group action $\bar{\Phi}_n:\hat{B}_{\tilde{A}_n}\rightarrow \mathrm{Autoeq}(D^b(\operatorname{\overline{\mathit{R}}_n-mod}))$ \cite[Remark 3.5]{GTW}. The notation is different  
from \cite{GTW} (our $\overline{R}_{n}$ and $\overline{\mathit{P}}_i$ are $R_n$ and $_iP$ of their paper).

\subsection{The faithfulness of the categorical action}

Following the idea of Seidel and Thomas \cite{ST}, we prove the faithfulness of the weak categorical action on the Calabi-Yau-$N$ category of the quiver algebra by transferring from the faithfulness result of the zig-zag quiver algebra by Gadbled, Thiel and Wagner \cite{GTW}.
\begin{lem}\cite[Lemma 5.7]{GTW} For $g\in\hat{B}_{\tilde{A}_n}$,  let $f\in \mathrm{MCG}(\mathbb{D}^*;\Delta)$ be a mapping class corresponding to $g$, and $\overline{L}_{n}^{g}$ represents $\bar{\Psi}_n(g)$. Then
	\begin{equation}\label{gin}
		\sum_{(r_1,r_2,r_3) \in \mathbb{Z}^3}\mathrm{dim}_{k}\mathrm{Hom}_{D^{b}(\operatorname{\overline{\mathit{R}}_n-mod})}(\overline{P}_i,\overline{L}_{n}^{g}(\overline{P}_j))[r_1]\{-r_2\}\langle-r_3\rangle)=2I(b_i,f(b
		_j))
	\end{equation}
	for all $1\leq i,j\leq n$. 
\end{lem}

Let $K^{b}_{f}(\operatorname{\overline{\mathit{R}}_n-proj})$ be a full subcategory of $K^b(\operatorname{\overline{\mathit{R}}_n-mod})$ whose objects are finite complexes in $\overline{P}_i$.

\begin{lem}\label{inter}
	For $g\in\hat{B}_{\tilde{A}_n}$, let $f\in \mathrm{MCG}(\mathbb{D}^*;\Delta)$ be a mapping class corresponding to $g$, and $L^g_{n,N}$ a functor which represents $\Phi_{n,N}(g)$. Then
	\begin{equation*}
		\sum_{s_1,s_2\in \mathbb{Z}}\mathrm{dim}_{k}\mathrm{Hom}_{D^b(\operatorname{\mathit{R}_n-mod})}(P_i,L^{g}_{n,N}(P_j)[s_1]\langle s_2\rangle)=2I(b_i,f(b_j))
	\end{equation*}
	for all $1\geq i,j\geq n$.
\end{lem}
\begin{proof}
	Consider a functor from $\operatorname{\overline{\mathit{R}}_n-proj}$ to the category of dgm over $R_n$ which sends the object $\overline{P}_i\{r_1\}\langle r_2\rangle$ to the dgm $P_i[\sigma_i-Nr_1]\langle r_2\rangle$, where $\sigma_i=-d_1-d_2-\cdots-d_{i-1}$. 
	Let $\overline{R}_n^{d_1,d_2}$ be the space of elements of the first grading of degree $d_1$ and the second grading of degree $d_2$ of $\overline{R}_n$. The elements of the subspace $(\overline{j})\overline{R}_n^{r_1-r_2,s_1-s_2}(\overline{i})$ is the the homomorphism of graded modules	
	$\overline{\mathit{P}}_i\{r_1\}\langle
	s_1\rangle\rightarrow\overline{\mathit{P}}_j\{r_2\}\langle s_2\rangle$. For the $R_n$ algebra counterpart, elements of the first grading of degree $\sigma_j-\sigma_i-N(r_2-r_1)$ and the second grading of degree $s_1-s_2$ in $(j)R_{n,N}(i)$ would correspond to dgm homomorphisms between
	$P_i[\sigma_i-Nr_1]\langle s_1\rangle$ and $P_j[\sigma_j-Nr_2]\langle s_2\rangle$.
	
	Sending the basis element $(\overline{i_0|\cdots|i_{\nu}})\in\overline{R}_n$ to the corresponding element to $(i_0|\cdots|i_{\nu})\in R_n$ gives an isomorphism for any $1\geq i,j\geq n$ and $d_1,d_2\in \mathbb{Z}$,
	\begin{equation*}
		\,(\overline{j})\overline{R}_n^{d_1,d_2}(\overline{i})\cong(j)R_n^{\sigma_j-\sigma_i+Nd_1,d_2}(i).
	\end{equation*}	
	Following the same argument of lemma 4.18 of \cite{ST}, we can construct an exact functor $\Pi:K^{b}_{f}(\operatorname{\overline{\mathit{R}}_n-proj})\rightarrow D(\operatorname{\mathit{R}_n-mod})$ such that
	\begin{equation}\label{mor}
		\bigoplus_{r_2=Nr_1} \mathrm{Hom}_{K^{b}_{f}(\operatorname{\overline{\mathit{R}}_n-proj})}(X,Y\{r_1\}[r_2]\langle r_3\rangle)\longrightarrow \mathrm{Hom}_{D(\operatorname{\mathit{R}_n-mod})}(\Pi(X),\Pi(Y)),
	\end{equation}
	is an isomorphism for all $X,Y\in K^{b}_{f}(\operatorname{\overline{\mathit{R}}_n-proj})$.
	
	As in the proof of lemma 4.17 of \cite{ST} the functor $\Pi$ commutes with functors $\overline{L}_n^{g}$, $L_{n,N}^{g}$ representing $\bar{\Psi}_n(g)$ and $\Phi_{n,N}(g)$ respectively. Using the equation (\ref{gin}) and the isomorphism (\ref{mor}), we have the following identity	
	
	\begin{align*}
		&\sum_{s_1,s_2}\mathrm{dim}_{k}\mathrm{Hom}_{D(\operatorname{\mathit{R}_n-mod})}(P_i,L^{g}_{n,N}(P_j)[s_1]\langle s_2\rangle)\\&=\sum_{s_1,s_2}\mathrm{dim}_{k}\mathrm{Hom}_{D(\operatorname{\mathit{R}_n-mod})}(\Pi(\overline{P}_i),\Pi\overline{L}^{g}_{n,N}(\overline{P}_j)[s_1]\langle s_2\rangle)\\&
		=\sum_{r_1,r_2,r_3}\mathrm{dim}_{k}\mathrm{Hom}_{D^{b}(\operatorname{\overline{\mathit{R}_{n}}-mod})}(\overline{P}_i,\overline{L}_{n}^{g}(\overline{P}_j))[r_1]\{-r_3\}\langle-r_2\rangle)\\&=2I(b_i,f_{\sigma}(b_j)).
	\end{align*}	
\end{proof}	
\begin{thm}\label{maint}
	Let $L_{n,N}^{g}$ be a functor representing $\Phi_{n,N}(g)$ for some $g\in\hat{B}_{\tilde{A}_n}$. If $L_{n,N}^{g}(P_j)\simeq P_j$ for all $j$, then g must be the identity element.
\end{thm}	
\begin{proof}	
	For $g\in\hat{B}_{\tilde{A}_n}$, choose $f$ and $L_{n,N}^g$ as in the lemma \ref{inter}.
	Take another element $g' \in\hat{B}_{\tilde{A}_n}$ and corresponding $f'$, $L_{n,N}^{g'}$.
	Applying lemma \ref{inter} to $(g')^{-1}g$, we have
	\begin{align*}
		&I(f'(b_i),f(b_j))=I(b_i,(f')^{-1}f(b_j))\\&=\frac{1}{2}\sum_{s_1,s_2}\mathrm{dim}_{k}\mathrm{Hom}_{D(\operatorname{\mathit{R}_n-mod})}(P_i,(L^{g'}_{n,N})^{-1}L^g_{n,N}(P_j)[s_1]\langle s_2\rangle)\\
		&=\frac{1}{2}\sum_{s_1,s_2}\mathrm{dim}_{k}\mathrm{Hom}_{D(\operatorname{\mathit{R}_n-mod})}(P_i,(L^{g'}_{n,N})^{-1}(P_j)[s_1]\langle s_2\rangle)\\
		&=I(b_i,(f')^{-1}(b_j))=I(f'(b_i),b_j).
	\end{align*}
	By lemma \ref{mapping}, $g$ is the central element  of the form $\rho^{ n\nu}$ for some $\nu\in\mathbb{Z}$.
	By (\ref{shift}), $L^{g}_{n,N}(P_j)\simeq P_j\langle-\nu\rangle$ and the assumption $L_{n,N}^{g}(P_j)\cong P_j$ implies that $\nu=0$ so $g=1$.
\end{proof}

In this subsection we show that the spherical twist subgroup $\mathrm{Sph}(\tilde{A}_{n,N})$ is isomorphic to the braid group $B_{\tilde{A}_n}$ of the affine type $A$. This statement was proved in \cite{Qiu} for $N=3$, so we will consider $N\ge 4$. We define the functor $\Psi_S$ to be the composition
\begin{equation*}
	\xymatrix{ 
		D(\tilde{A}_{n,N})\ar[rr]^{hom(S,-)}&& K(end(S))\ar[rrr]^{ \text{quotient functor}}&&&D(end(S))
	}.
\end{equation*}
The twist functors $\mathrm{Tw}_{S_i}$ on $D(\tilde{A}_{n,N})$ and $\mathcal{T}_i$ on $D(end(S))\simeq D^b(\operatorname{\mathit{R}_{n}-mod})$ are related through $\Psi_S$ \cite[lemma 4.3]{ST}.

\begin{lem}\label{com} Let the objects $S_1,\cdots S_n$ be as above. The twist functors $\mathrm{Tw}_{S_i}$ are defined in definition 2.2. Then the following diagram is commutative up to isomorphism for $1\leq i \leq n$:
	\begin{equation}
		\xymatrix{ 
			D(\tilde{A}_{n,N}) \ar[rr]^{\mathrm{Tw}_{S_i}}\ar[d]_{\Psi_S} && D(\tilde{A}_{n,N})\ar[d]^{\Psi_S} 
			&\\
			D(end(S)) \ar[rr]_{\mathcal{T}_i} && D(end(S))
		}
	\end{equation}
	
\end{lem}
An augmented graded algebra is a graded algebra $R$ over a ring $B$ with a graded algebra homomorphism $\epsilon_R:R\rightarrow B$ and $\iota_A:B\rightarrow R$ satisfies $\epsilon_R\circ\iota_R=id_{B}$. We write $R^+$ for the subspace of elements of positive degree.
We recall the result in \cite{ST} that the cohomology of an augmented graded algebra $R_{n}$ determines the quasi-isomorphism type so that we can construct a chain of quasi-isomorphims connecting differential graded algebra $R_{n}$ and endomorphism algebra $end(S)$. 
\begin{defi}
	A differential graded algebra $R$ is called formal if its cohomology algebra $H(A)$ is isomorphic to itself.
	A graded algebra is called intrinsically formal if any dga $C$ with $H(C)\cong R$ is formal.
\end{defi}

\begin{thm} \cite[Theorem 4.7]{ST}
	Let A be an augmented graded algebra. If the Hochschild cohomology $HH^q(A,A[2-q])=0$ for all $q>0$, then A is intrinsically formal.	
\end{thm}
\begin{lem}\label{int}
	$R_n$ is intrinsically formal for $N\geq 4$ and $n\geq 3$.
\end{lem}

\begin{proof}
	The degree for any path of the cyclic quiver $\Gamma_n$ is $\geq[N/2]$. The degree of two composed paths would be $N$. This implies that the degree of any non-zero path $(i_0|\cdots|i_m)$ of length $m$ in $\mathbb{K}\Gamma_n$ is $\geq[(Nm/2)]$. Any elements of $(R_{n}^+)^{\otimes_Bq}$ is of the form
	\begin{equation*}
		c=(i_{1,0}|\cdots|i_{1,m_1})\otimes(i_{2,0}|\cdots|i_{2,m_2})\otimes\cdots(i_{q,0}|\cdots|i_{q,m_q})
	\end{equation*}
	with all $m_q>0$
	Because the tensor product is over $B=k^{\oplus n}$, such a $c$ can be nonzero only if the paths match up, $i_{p,m_p}=i_{p+1,0}$. Therefore,
	\begin{equation*}
		\mathrm{deg}(c)=\mathrm{deg}(i_{1,0}|\cdots|i_{1,m_1}|i_{2,1}|\cdots|i_{2,m_2}|i_{3,1}|\cdots|i_{q,m_q})\geq [N(m_1+\cdots m_q)/2]
	\end{equation*}
	So $(R_{n}^+)^{\otimes_Bq}$ is concentrated in degree $\geq [(Nq)/2]$, while $R_{n}[2-q]$ is concentrated in degrees $\leq N+q-2$. 
	Hochschild cochain group is zero,
	\begin{equation*}
		C^q(R_{n},R_{n}[2-q])=\mathrm{Hom}_{B-B}(R_{n}^+,R_{n}[2-q])=0
	\end{equation*}
	for $N\geq 4$.

\end{proof}

\begin{thm}\label{faithful}
	Suppose $N\geq 4$ and $n\geq 3$, let $S_1,\cdots, S_n$ be simple $\mathbb{K}\Gamma_N\tilde{A}_n$ modules. The corresponding spherical twist functors $\mathrm{Tw}_{S_1},\cdots\mathrm{Tw}_{S_n}$	generate a homomorphism $\Phi:B_{\tilde{A}_n}\rightarrow \mathrm{Aut}(D(\tilde{A}_{n,N}))$ which is injective.
\end{thm}

\begin{proof}
	The spherical twist functors satisfy the braid group relations following proposition 2.12 and 2.13 in \cite{ST}: \begin{align*}
		\mathrm{Tw}_{S_i}\mathrm{Tw}_{S_j}&\simeq \mathrm{Tw}_{S_j}\mathrm{Tw}_{S_i}\quad\,\,\quad\quad\quad\quad for\,|i-j|\geq 2\\
		\mathrm{Tw}_{S_i}\mathrm{Tw}_{S_{i+1}}\mathrm{Tw}_{S_i}&\simeq \mathrm{Tw}_{S_{i+1}}\mathrm{Tw}_{S_i}\mathrm{Tw}_{S_{i+1}}\quad\,\,for\,\, i=1,\cdots,n.
	\end{align*} 
	So the twist functors $\mathrm{Tw}_{S_i}$ generate the homomorphism $\Phi$.

	Shifting each $S_i$ by some amount, $\mathrm{Hom}^*(S_{i+1},S_i)$ is concentrated in degree $d_i$ for $i=1,\cdots,n$. We take resolutions $S'_1,\cdots S'_n\in D(\tilde{A}_{n,N})$ for simple modules $S_1,\cdots S_n$.
	By lemma \ref{coho}, we have $H(end(S'))\simeq R_{n}$. By the intrinsic formality of $R_n$, lemma \ref{int} implies that $end(S')$ is quasi-isomorphic to $R_{n}$. We define a functor $\Psi$ to be the composition
	\begin{equation*}
		\xymatrix{ 
			D(\tilde{A}_{n,N}) \ar[r]^{\Psi_{S'}} & D(end(S'))  \ar[r]^{\cong} & D(\operatorname{\mathit{R}_n-mod}).
		}
	\end{equation*}	
	We have $\Psi(S_i)=P_i$ for $i=1,\cdots,n$. Consider the diagram
	\begin{equation*}
		\xymatrix{ 
			D(\tilde{A}_{n,N}) \ar[d]_{\mathrm{Tw}_{S_i}} \ar[rr]^{\Psi_{S'}} && 
			D(end(S')) \ar[d]^{T_i} \ar[rr]^{\cong}&&  D(\operatorname{\mathit{R}_n-mod})\ar[d]^{T_i} \\ 
			D(\tilde{A}_{n,N})\ar[rr]^{\Psi_{S'}} && D(end(S'))\ar[rr]^{\cong}&&D(\operatorname{\mathit{R}_n-mod})
		}
	\end{equation*}
	The first square commutes by lemma \ref{com} and the second one by Lemma 4.2 in \cite{ST}.
	Let $g$ be an element of $B_{\tilde{A}_n}$, a functor $R^{g}:D(\tilde{A}_{n,N})\rightarrow D(\tilde{A}_{n,N})$ which represents $\Phi(g)$, and $R_{n,N}^{g}:D(\operatorname{\mathit{R}_n-mod})\rightarrow D(\operatorname{\mathit{R}_n-mod})$ which represents $\rho_{n,N}(g)$.
	By the above commutative diagram, we have
	\begin{equation*}
		R^g_{n,N}\circ \Psi\cong \Psi\circ R^{g}.
	\end{equation*}
	Assume $R^g(S_i)\cong S_i$ for all $i$; $R_{n,N}^{g}(P_i)=R_{n,N}^{g}\Psi(S_i)\cong\Psi R^{g}(S_i)\cong\Psi(S_i)\cong P_i$. By theorem \ref{maint}, $g$ is the identity.
\end{proof}

The correspondence of generators $\sigma_i\rightarrow \mathrm{Tw}_{S_i}$ gives an isomorphism of groups $B_{\tilde{A}_n}\simeq \mathrm{Sph}(\tilde{A}_{n,N})$.

\section{Quadratic differentials and period maps}

First, we explain the Riemann surface with the meromorphic quadratic differential related to the affine type $A$ quiver algebra and construct their central charge formula for the stability conditions as a period map. 
Consider a meromorphic quadratic differential
\begin{equation*}
	\phi(z)=\bigg[\frac{\prod_{i=1}^{n}(z-a_i)}{z^{p}}\bigg]^{N-2}\frac{dz^{\otimes 2}}{z^2} 
\end{equation*}
which has zeroes of order $(N-2)$ at distinct $n$ points and a pole of order $p(N-2)+2$ at the origin and a pole of order $(N-2)(n-p)+2$ at the $\infty\in\mathbb{P}^1$. Denote by $\mathrm{Quad}(N,n)$ the space of quadratic differential $\phi(z)$.
In this section we review some terminologies of the geometry of quadratic differentials defined in \cite{Ike16} and fix notations. 
\subsection{Homology groups and Periods}
We define homology groups $H_{+}(\phi)$ for $N$ even 
\begin{equation*}
	H_+(\phi):=H_1(\mathbb{C},\mathrm{Zero}(\phi);\mathbb{Z}).
\end{equation*}
Let S be the spectral cover of $\mathbb{P}^1$ defined by $\phi$
\begin{equation*}
	S=\{(p,l(p))\,,p\in \mathbb{P}^1\,|\,\,l(p)\in L_p\,\,\text{such that}\,l(p)\otimes l(p)=\phi(p) \},
\end{equation*}
where $L$ is the total space of the line bundle $\omega_{\mathbb{P}^1}(E)$. Let $\pi: S\rightarrow\mathbb{P}^1$ be an branched covering map. The homology group $H_-(\phi)$ for $N$ odd is defined by
\begin{equation*}
	H_{-}(\phi):=H_{1}(S\,\backslash\,\pi^{-1}(\infty)\cup\pi^{-1}(0);\mathbb{Z}).
\end{equation*}
The intersection form on homology groups is defined in \cite{Ike16}
\begin{equation*}
	I_{\pm}:H_{\pm}(\phi)\times H_{\pm}(\phi)\rightarrow \mathbb{Z}.
\end{equation*} 
We can define a period map of $\phi$, $Z_{\phi}:H_{\pm}(\phi)\rightarrow\mathbb{C}$ by 
\begin{equation*}
	Z_{\phi}(\alpha)=\int_{\alpha}\sqrt{\phi}\in\mathbb{C} 
\end{equation*}
for any cycle $\alpha\in H_{\pm}(\phi)$.

Let $\Gamma$ be a free abelian group of rank $n$. The $\Gamma$-framing of $\phi\in \mathrm{Quad}(N,n)$ is an isomorphism of abelian groups
\begin{equation*}
	\psi: \Gamma\rightarrow H_{\pm}(\phi).
\end{equation*} 

Denote by $\mathrm{Quad}(N,n)^{\Gamma}$ the space of pairs $(\phi,\psi)$ consisting of a quadratic differential $\phi\in \mathrm{Quad}(N,n)$ and a $\Gamma$-framing $\psi$.
For $(\phi,\psi)\in \mathrm{Quad}(N,n)^{\Gamma}$, a period map is defined by the composition of a framing and $Z_{\phi}$,
\begin{align}\label{per}
	\mathcal{W}_N: \mathrm{Quad}(N,n)^{\Gamma}&\rightarrow \mathrm{Hom}_{\mathbb{Z}}(\Gamma,\mathbb{C})\\(\phi,\psi)\quad&\mapsto Z_{\phi}\circ \psi
\end{align}
which is a local isomorphism of complex manifold by \cite[theorem 4.12]{BrSm}.

The homology groups form a local system over the orbifold $\mathrm{Quad}(N,n)^{\Gamma}$.
We take a continuous path connecting two quadratic differentials $\phi$ and $\phi'$ in a contractible subset in $\mathrm{Quad}(N,n)$. A parallel transport along this path gives the Gauss-Manin connection which identifies the fibres $H_{\pm}(\phi)$ and $H_{\pm}(\phi')$.

For a framed differential $(\phi,\psi)\in \mathrm{Quad}(N,n)^{\Gamma}$, we denote by $\mathrm{Quad}(N,n)^{\Gamma}_*$ the connected component containing $(\phi,\psi)$. We consider a path $c$ connecting $\phi$ to any point $\phi'$ in $\mathrm{Quad}(N,n)^{\Gamma}_*$ and the new framing is given by $\psi'=\mathrm{GM}_c\circ\psi$ such that the diagram

\begin{equation}\label{gm}
	\xymatrix{ 
		& \Gamma \ar[dl]_{\psi} \ar[dr]^{\psi'}  & \\ 
		H_{\pm}(\phi) \ar[rr]^{\mathrm{GM}_c} && H_{\pm}(\phi'),
	}
\end{equation}
commutes.

\subsection{Trajectories}
At a point in $\mathbb{P}^1\backslash \mathrm{Crit}(\phi)$, there is a distinguished local coordinate $w$ for the differential on $\mathbb{P}^1$
\begin{equation*}
	w=\int\sqrt{\phi(z)}dz.
\end{equation*}
Then the differential has locally the form in this coordinate
\begin{equation*}
	\phi(w)=dw\otimes dw.
\end{equation*} 

We define a foliation of the phase $\xi$ on $\mathbb{P}^1\backslash \mathrm{Crit}(\phi)$ determined by the equation
\begin{equation}\label{folia}
	\mathrm{Im}(w/e^{\sqrt{-1}\pi\xi})=\mathrm{constant}.
\end{equation}
A horizontal foliation is a foliation of the phase $\xi=0$ of the above equation.
A straight arc of phase $\xi$ is an integral curve $I\rightarrow\mathbb{P}^1\backslash \mathrm{Crit}(\phi)$ satisfying the equation (\ref{folia}) for an open interval $I\in \mathbb{R}$.
Maximally extended straight arcs can be classified as follows:
\begin{itemize}
	\item[(1)] A \emph{saddle connection} of phase $\xi\neq 0$ is a maximal straight arc approaching two distinct zeroes of $\phi$.
	\item[(2)] A \emph{saddle trajectory} is a maximal straight arc of phase $\xi=0$ approaching two distinct zeroes of $\phi$.
	\item[(3)] A \emph{separating trajectory} is a maximal straight arc of phase $\xi=0$ approaching a zero at one end and a pole at the other end of $\phi$.
	\item[(4)] A \emph{generic trajectory} is a maximal straight arc of phase $\xi=0$ approaching a pole at both ends.
\end{itemize} 
A differential $\phi\in \mathrm{Quad}(N,n)$ is called \emph{saddle-free} if there is no saddle trajectory of $\phi$. For a saddle-free differential $\phi$, the maximal set of separating trajectories splits $\mathbb{P}^1$ into a disjoint union of connected components, horizontal strips $\{w\in\mathbb{C}|\,a<\mathrm{Im}w<b\}$ and half planes $\{w\in\mathbb{C}|\,\mathrm{Im}w>0\}$. There are $n$ horizontal strips and $(N-2)p$ half planes in the neighborhood of $0$ and $(N-2)q$ half planes in the neighborhood of $\infty$. Thus, we have $n$ saddle connections and its hat-homology classes form a basis of $H_{\pm}(\phi)$. For details, we refer to section 3 of \cite{Ike16}.

Next, we consider the behavior near a zero. 
Let $\phi $ and $p \in\mathbb{P}^1$ be a zero of order $N-2$. Then 
there is a local coordinate $t$ of a neighborhood $p \in U \subset \mathbb{P}^1$ ($t=0$ corresponds to $p$) 
such that 
\begin{equation*}
	\phi=c^2\, t^{N-2} dt^{\otimes 2}, \quad c=\frac{1}{2}\,N.
\end{equation*}
Hence on $U$, the distinguished local coordinate $w$ takes the 
form $w=t^{N/2}$. There are $N$ horizontal straight arcs coming from $p$
so the differential $\phi$ gives us $N$-angulations. 

For poles at $0$ and $\infty$ of $\mathbb{P}^1$, the differential $\phi$ has a pole of order $d_{p}^{N}+2=
p(N-2)+2$ at $0\in \mathbb{P}^1$ and a pole of order $d_{q}^{N}+2=(n-p)(N-2)+2$ at $\infty \in \mathbb{P}^1$.  
There is a neighborhood $0\in U$ ($\infty \in V$) $\subset \mathbb{P}^1$ and a collection of $d_{p}^{N}$ ($d_{q}^{N}$) distinguished 
tangent directions $v_i\,(i=1,\dots,d_{p}^{N})$ at 0 \big($v_j\,(j=1,\dots,d_{q}^{N}\big)$ at $\infty$) such that any trajectory entering $U(V)$ tends to 
0 ($\infty$) and becomes asymptotic to one of the $v_i$ ($v_j$). 

\section{Calabi-Yau $N$-categories from $N$-angulations of surfaces}
\subsection{The $N$-angulation of the marked annulus and colored quivers}
We take the real oriented blowup at $0$ and $\infty$ of $\mathbb{P}_1$ by replacing $0$ and $\infty$ with $S^1$ which corresponds to tangent directions at 0 and $\infty$, then we obtain an annulus with $d_{p}^{N}=
p(N-2)$ points on the interior boundary and $d_{q}^{N}=(n-p)(N-2)$ points on the exterior boundary. Those marked points correspond to the tangent direction at $0$ and $\infty$, respectively. Denote the marked annulus by $P_{d_{p}^{N},d_{q}^{N}}$ which is a regular $d_{q}^{N}$-gon cut out by an regular $d_{p}^{N}$-gon at its center. 

Label the vertices on the exterior polygon in the counterclockwise direction $i=0,\cdots p-1$ and label vertices of the interior polygon in the clockwise direction $i=0,\cdots ,q-1$.
An $N$-diagonal is a path connecting two vertices of the interior polygon or the exterior polygon. There are three types of path with restrictions:
\begin{itemize}
	\item[Type 1] A path connecting a vertex $i$ of the external polygon and a vertex $j$ of the interior polygon in $P_{d_{p}^{N},d_{q}^{N}}$. In addition, $i$ is congruent to $j$ modulo $N-2$. 
	\item[Type 2] A path from the vertex $i$ to the vertex $i+(N-2)k+1$ of the exterior polygon for $k\geq 0$
	\item[Type 3] A path from the vertex $i$ to the vertex $i+(N-2)k+1$ of the interior polygon for $k\geq 0$.
\end{itemize}
\begin{defi}
	An $N$-angulation $\Delta$ of $P_{d_{p}^{N},d_{q}^{N}}$ is a maximal set of non-crossing $N-2$ diagonals that divides the marked annulus into $N$-gons.
\end{defi}

\begin{figure}
	\begin{center}
		\tikzstyle{dot}=[circle,draw=black!100,fill=black!100,thick,
		inner sep=0pt,minimum size=2mm]
		\tikzstyle{sdot}=[circle,draw=black!100,fill=black!100,thick,
		inner sep=0pt,minimum size=0.8mm]
		\begin{tikzpicture}[scale=2.5]
			\draw (0,0) circle (1cm);
			\node (1) at (-0.28,0.96) [dot] {};	
			\node (2) at (0.28,0.96) [dot] {};
			\node (3) at (0.8,0.6) [dot] {};
			\node (4) at (1,0) [dot] {};
			\node (5) at (0.8,-0.6) [dot] {};
			\node (6) at (0.38,-0.923) [dot] {};
			\node (7) at (-0.28,-0.96) [dot] {};
			
			\node (8) at (-0.6,-0.8) [dot] {};
			\node (9) at (-0.923,-0.38) [dot] {};
			\node (10) at (-1.0,-0.0) [dot] {};
			\node (11) at (-0.923,0.38) [dot] {};
			
			\node[dot] (12) at (-0.6,0.8)  {};
			
			\draw (0,0) circle (0.4cm);
			\node (1) at (-0.28,0.96) [dot] {};	
			\node (2) at (0.28,0.96) [dot] {};
			
			\node (a) at (-0.24,0.32) [dot] {};	
			\node (b) at (0.24,0.32) [dot] {};		
			\node (f) at (-0.4,0) [dot] {};	
			\node (c) at (0.4,0) [dot] {};		
			\node (e) at (-0.24,-0.32) [dot] {};	
			\node (d) at (0.24,-0.32) [dot] {};		
			
			\draw[very thick,blue] (8).. controls (-0.78,0.0) .. (12);
			\draw[very thick,blue] (a)--(12);
			\draw[very thick,blue] (a).. controls (0.0,0.5) ..(3);
			\draw[very thick,blue] (a).. controls (-1.3,-0.2) and (0.55,-1.2) ..(c);
			
			\draw[very thick,blue] (b).. controls (0.6,0.5) and (1.2,-0.8) ..(8);
			\draw[very thick,blue] (3).. controls (0.8,-0.6) ..(7);
		\end{tikzpicture}
		\caption{An example of a 5-angulation of an annulus }
	\end{center}
\end{figure}
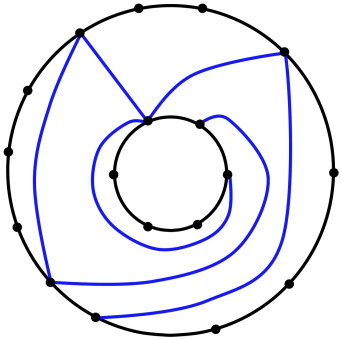

The number of $(N-2)$ diagonals of any $N$-angulation of  $P_{d_{p}^{N},d_{q}^{N}}$ is exactly $p+q$ and there is at least one $N-2$-diagonal connecting a vertex of the exterior polygon and a vertex of the interior polygon. The foliation of $\phi$ gives rise to a horizontal strip decomposition of the annulus. For a saddle-free differential $\phi\in \mathrm{Quad}(N,n)$, by taking one generic trajectory from each horizontal strip in the decomposition of by $\phi$, we have an $N$-angulation $\Delta_{\phi}$ of $P_{d_{p}^{N},d_{q}^{N}}$. An example of an $N$-angulation is shown in figure 4.

We define a colored quiver $Q_{\Delta}$ with $p+q$ vertices
associated to an $N$-angulation $\Delta$ of $P_{d_{p}^{N},d_{q}^{N}}$ \cite{Tor12}. The vertices of the quiver correspond to the $N-2$-diagonals. There is an arrow between $i$ and $j$ if the $N-2$-diagonals bound a common
$N$-gon. The colour of the arrow $(c)$ is the number of edges forming
the segment of the boundary of the $N$-gon which lies between $i$
and $j$, counterclockwise from $i$.
See figure 5 for an example.

An $(N-2)$-colored quiver $Q$ consists of vertices $\{1,\cdots,n\}$ and colred arrows $i\xrightarrow{(c)}j$, where $c\in\{0,1,\cdots,N-2\}$. Denote by $q^{(c)}_{ij}$ the number of arrows from $i$ to $j$ of color $(c)$.
we assume the following conditions:
\begin{itemize}
	\item[(1)] No loops, $q^{(c)}_{ii}$ for all $c$.
	\item[(2)] Monochromaticity, if $q_{ij}^{(c)}\neq 0$, then $q^{(c')}=0$ for $c\neq c'$.
	\item[(3)] Skew-symmetry, $q^{(c)}_{ij}=q^{(N-2-c)}_{ji}$.
\end{itemize}
Let $Q$ be a $(N-2)$ colored quiver satisfying the above conditions and fix a vertex $v$ in $Q$.
We define the forward colored quiver mutation of $Q$ at vertex $v$:
\begin{align*}
	\tilde{q}^{(c)}_{ij}=
	\begin{cases}
		q^{(c-1)}_{ij}  &\text{if} \quad v=i \\
		q^{(c+1)}_{ij}  &\text{if} \quad v=i \\
		\text{max}\{0,q^{(c)}_{ij}-\sum_{t\neq c}q^{(t)}_{ij}+(q^{(c)}_{iv}-q^{(c+1)}_{iv})q^{(0)}_{vj}\\
		\quad\quad\quad\quad\quad\quad\quad\quad\quad+q^{(N-2)}_{vj}(q^{(c)}_{vj}-q^{(c-1)}_{vj})\}\quad &\text{otherwise}.
	\end{cases}
\end{align*}
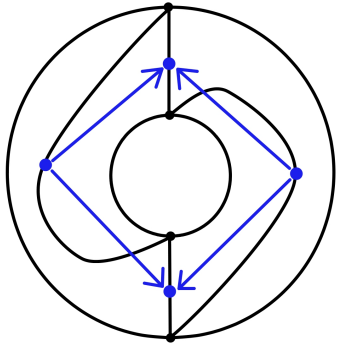
\begin{figure}
	\begin{center}
		\tikzstyle{post}=[<-,shorten <=1pt,>=stealth',very thick,draw=orange!100,]	
		\tikzstyle{dot}=[circle,draw=black!100,fill=black!100,thick,
		inner sep=0pt,minimum size=2mm]
		\tikzstyle{v}=[circle,draw=orange!100,fill=orange!100,thick,
		inner sep=0pt,minimum size=3mm]
		\begin{tikzpicture}[scale=2.5]
			\draw (0,0) circle (1cm);
			\draw (0,0) circle (0.4cm);
			\node (1) at (0,1) [dot] {};	
			\node (2) at (0,-1) [dot] {};		
			\node (A) at (0,0.4) [dot] {};	
			\node (B) at (0,-0.4) [dot] {};	
			\draw[very thick,blue] (1)--(A);
			\draw[very thick,blue] (2)--(B);
			\draw[very thick,blue] (1).. controls (-1.3,-0.2) and (-0.55,-0.8) ..(B);
			\draw[very thick,blue] (2).. controls  (1.3,0.2)and(0.55,0.8) ..(A);
			\node (p) at (0.7,0) [v] {};
			\node (q) at (-0.7,0) [v] {};
			\node (s) at (0,0.7) [v] {}
			edge[post,<-,](p)
			edge[post,<-,](q);
			\node (t) at (0,-0.7) [v] {}
			edge[post,<-,](p)
			edge[post,<-,](q);
		\end{tikzpicture}
		\caption{An example of a triangulation of $P_{2,2}$ and its corresponding quiver $Q_{\Delta}$ in orange color.}
	\end{center}
\end{figure} 	

\subsection{Equivalence between the mutation class of $N$-angulation and hearts}

We define the mutation of $N$-angulation in the following. Let $\Delta$ be any $N$-angulation of $P_{d_{p}^{N},d_{q}^{N}}$ and $\delta$ be any $N-2$-diagonal in the $\Delta$. There are two $N$-gons with the common edge $\delta$. By removing the $N-2$-diagonal from $\Delta$, we have an $2N-2$-gon in the almost complete N-angulation $\Delta-\delta$. There are $N-1$ possible diagonals $\alpha_i$ such that $\Delta-\delta\cup \alpha_i$ is an $N$-angulation. Those possible diagonals will be called diameters of the $(N-2)$-gon.

Denote by $\mu_{\delta}^{\sharp}\Delta$ the forward flip with respect to the diagonal $\delta$ obtained by rotating two boundary points of $\delta$ clockwise one step. Then we have obtained a new $(N-2)$-diagonals $\delta^{\sharp}$ and a new $N$-angulation $\mu^{\sharp}_{\delta}\Delta=(\Delta\backslash\delta)\cup\delta^{\sharp}$. We also define the backward flip operation with respect to $\delta$ by rotating counterclockwise and obtain a new $N-2$-diagonal $\delta^{\flat}$ and new $N$-angulation $\mu^{\flat}_{\delta}\Delta=(\Delta\backslash\delta)\cup\delta^{\flat}$. Mutation at any $N-2$-diagonal $\Delta$ of $P_{d_{p}^{N},d_{q}^{N}}$ corresponds to the colored quiver mutation at $v_{\alpha}$ in $Q_{\Delta}$ \cite[proposition 5.1]{Tor12}. See figure 6 for an example.

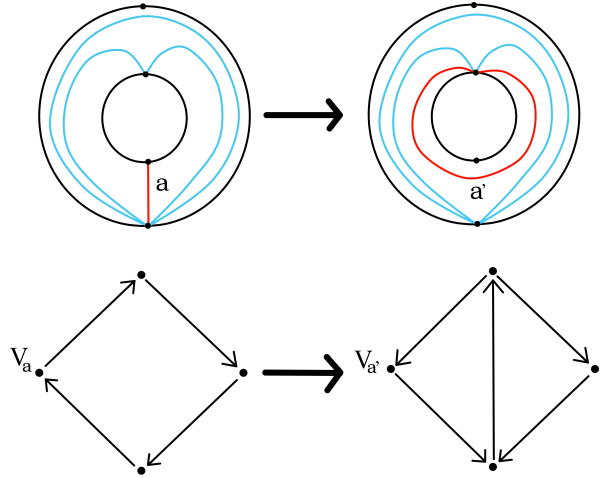
\begin{figure}
	\tikzstyle{post}=[<-,shorten <=1pt,>=stealth',very thick,draw=orange!100,]	
	\tikzstyle{dot}=[circle,draw=black!100,fill=black!100,thick,
	inner sep=0pt,minimum size=2mm]
	\tikzstyle{v}=[circle,draw=orange!100,fill=orange!100,thick,
	inner sep=0pt,minimum size=3mm]
	\centering
	\begin{tikzpicture}[scale=2]
		\coordinate (shift) at (1.5,0);
		\begin{scope}[shift=(shift))]
			\centering
			\draw (0,0) circle (1cm);
			\draw (0,0) circle (0.4cm);
			\node (1) at (0,1) [dot] {};	
			\node (2) at (0,-1) [dot] {};		
			\node (A) at (0,0.4) [dot] {};	
			\node (B) at (0,-0.4) [dot] {};	
			\draw[thick,red] (2)--(B);
			\tikzset{
				show curve controls/.style={
					decoration={
						show path construction,
						curveto code={
							\draw [blue, dashed]
							(\tikzinputsegmentfirst) -- (\tikzinputsegmentsupporta)
							node [at end, cross out, draw, solid, red, inner sep=2pt]{};
							\draw [blue, dashed]
							(\tikzinputsegmentsupportb) -- (\tikzinputsegmentlast)
							node [at start, cross out, draw, solid, red, inner sep=2pt]{};
						}
					}, decorate
				}
			}

			\draw [blue] plot [smooth cycle] coordinates {(0,-1)(-0.6,-0.5) (-0.8,0)(-0.6,0.5) (0,0.8) (0.6,0.5)(0.8,0)(0.6,-0.5) };
			\draw [blue] plot [smooth, tension=2] coordinates { (0,-1) (-0.6,0) (0,0.4)};
			\draw [blue] plot [smooth, tension=2] coordinates { (0,-1) (0.6,0) (0,0.4)};

			\draw (0,0) circle (1cm);
			\draw (0,0) circle (0.4cm);
			\node (1) at (0,1) [dot] {};	
			\node (2) at (0,-1) [dot] {};		
			\node (A) at (0,0.4) [dot] {};	
			\node (B) at (0,-0.4) [dot] {};	
			\centering	
		\end{scope}
		\draw [->,thick] (-0.4,0) -- (0.4,0)
		node [above,text width=3cm,text centered,midway]
		{
			flip at the\\ \textcolor{red}{diagonal} 
		};

		\coordinate (shift) at (-1.5,0);
		\begin{scope}[shift=(shift)]
			\draw (0,0) circle (1cm);
			\draw (0,0) circle (0.4cm);
			\node (1) at (0,1) [dot] {};	
			\node (2) at (0,-1) [dot] {};		
			\node (A) at (0,0.4) [dot] {};	
			\node (B) at (0,-0.4) [dot] {};	
			
			\tikzset{
				show curve controls/.style={
					decoration={
						show path construction,
						curveto code={
							\draw [blue, dashed]
							(\tikzinputsegmentfirst) -- (\tikzinputsegmentsupporta)
							node [at end, cross out, draw, solid, red, inner sep=2pt]{};
							\draw [blue, dashed]
							(\tikzinputsegmentsupportb) -- (\tikzinputsegmentlast)
							node [at start, cross out, draw, solid, red, inner sep=2pt]{};
						}
					}, decorate
				}
			}

			\draw [blue] plot [smooth cycle] coordinates {(0,-1)(-0.6,-0.5) (-0.8,0)(-0.6,0.5) (0,0.8) (0.6,0.5)(0.8,0)(0.6,-0.5) };
			\draw [blue] plot [smooth, tension=2] coordinates { (0,-1) (-0.6,0) (0,0.4)};
			\draw [blue] plot [smooth, tension=2] coordinates { (0,-1) (0.6,0) (0,0.4)};
			\draw [red] plot [smooth cycle] coordinates {(0,0.4)(-0.36,0.3)(-0.47,-0.12)(0,-0.6)(0.47,-0.12)(0.36,0.3)};

			\draw (0,0) circle (1cm);
			\draw (0,0) circle (0.4cm);
			\node (1) at (0,1) [dot] {};	
			\node (2) at (0,-1) [dot] {};		
			\node (A) at (0,0.4) [dot] {};	
			\node (B) at (0,-0.4) [dot] {};	
		\end{scope}
		\coordinate (shift) at (-1.5,-2.5);
		\begin{scope}[shift=(shift))]
			\tikzstyle{post}=[<-,shorten <=1pt,>=stealth',thick,draw=black!100,]	
			\tikzstyle{b}=[circle,draw=blue!100,fill=blue!100,thick,
			inner sep=0pt,minimum size=3mm]
			\tikzstyle{r}=[circle,draw=red!100,fill=red!100,thick,
			inner sep=0pt,minimum size=3mm]
			\node (p1) at (1,0) [b] {};
			\node (q1) at (-1,0) [r] {};
			\node (s1) at (0,1) [b] {}
			edge[post,<-,](q1);
			\node (t1) at (0,-1) [b] {}
			edge[post,<-,](p1);
			\node (p1) at (1,0) [b] {}
			edge[post,<-,](s1);
			\node (q1) at (-1,0) [r] {}
			edge[post,<-,](t1);
		\end{scope}
		
		\draw [->,thick] (-0.3,-2.5) -- (0.3,-2.5)
		node [above,text width=3cm,text centered,midway]
		{
			mutation at\\ the \textcolor{red}{vertex} 
		};

		\coordinate (shift) at (1.5,-2.5);
		\begin{scope}[shift=(shift))]
			\tikzstyle{post}=[<-,shorten <=1pt,>=stealth',thick,draw=black!100,]	
			\tikzstyle{b}=[circle,draw=blue!100,fill=blue!100,thick,
			inner sep=0pt,minimum size=3mm]
			\tikzstyle{r}=[circle,draw=red!100,fill=red!100,thick,
			inner sep=0pt,minimum size=3mm]
			\node (p) at (1,0) [b] {};
			\node (q) at (-1,0) [r] {};
			\node (s) at (0,1) [b] {};
			\node (t) at (0,-1) [b] {}
			edge[post,<-,](q)
			edge[post,<-,](p);
			\node (p) at (1,0) [b] {}
			edge[post,<-,](s);
			\node (q) at (-1,0) [r] {}
			edge[post,<-,](s);
			\node (s) at (0,1) [b] {}
			edge[post,<-,](t);
		\end{scope}

	\end{tikzpicture}
	\centering	
	
	\caption{An example of a flip of triangulations of annulus and the mutation at the vertex of the corresponding quivers}
\end{figure}
Define the graph $\tilde{\mathcal{G}}(P_{d_{p}^{N},d_{q}^{N}})$ to be an oriented graph whose vertices are all $N$-angulations of 
$P_{d_{p}^{N},d_{q}^{N}}$, and for any two $N$-angulations $\Delta$ and $\Delta^{\prime}$, the 
arrow $\Delta\rightarrow \Delta^{\prime}$ is given if 
there is an $(N-2)$-diagonal $\delta \in \Delta$ such that
$\Delta^{\prime}=\mu_{\delta}\Delta$, where $\mu_{\delta}$ is the flip on the diaginal $\delta$. Different ways of $N$-angulation can be related by flip operations in finite steps.

Let $\mathcal{M}_{d_p^N,d_q^N}$ be the mutation class of $(N-2)$-colored quivers of affine type $A$. There is a surjective function\begin{equation*}
	\sigma_{d_p^N,d_q^N}: \tilde{\mathcal{G}}(P_{d_{p}^{N},d_{q}^{N}}) \rightarrow \mathcal{M}_{d_p^N,d_q^N}
\end{equation*}
given by $\sigma_{d_p^N,d_q^N}(\Delta)=Q_{\Delta}$.
This function is not injective because of the following reason.
Define an operation $[s]$ on an $N-2$-diagonal by rotating the exterior polygon $s$ steps clockwise and the interior polygon $s$ steps counterclockwise. Let $\Delta$ be an $N$-angulation of $P_{d_{p}^{N},d_{q}^{N}}$. By applying $[s]$ on each diagonal in $\Delta$, we obtain a new $N$-angulation $\Delta[s]$. The corresponding colored quiver is the same $Q_{\Delta[s]}=Q_{\Delta}$ since $(N-2)$-diagonals bounding a common $N$-gons of $\Delta$ also bound a common  $N$-gons in $\Delta[s]$.

To establish a bijection between $N$-angulations and the corresponding colored quivers, we need to quotient $\tilde{\mathcal{G}}(P_{d_{p}^{N},d_{q}^{N}})$ by some relations \cite{Tor12}.
We define two functions $r_E$ and $r_I$ on the set of diagonals.
Let $\alpha$ be a diagonal of an $N$-angulation. Define $r_E (\alpha)$ to be the diagonal by rotating 
the exterior polygon one step clockwise.
Define $r_I(\alpha)$ to be the diagonal by rotating the interior polygon one step counterclockwise.
$r_E(\Delta)$ ($r_I(\Delta)$) be the new $N$-angulation obtained from $r_E$ ($r_I$) applying on each diagonal of $\Delta$.  

If $p\neq q$, two $N$-angulations $\Delta$ and $\Delta'$ to be equivalent if and only if $\Delta'=r_E^i r_I^j(\Delta) $ for some integer $i$ and $j$. If $p=q$, then $\Delta'=r_E^i r_I^j\epsilon^k(\Delta)$, where $\epsilon$ is an operation sending $P_{d_p^N,d_q^N}$ to $P_{d_q^N,d_p^N}$ by interchanging the interior polygon and the exterior polygon. The main result of \cite[Theorem 8.11]{Tor12} is the following:
\begin{thm}\label{equiv}
	Let $\tilde{\mathcal{G}}(P_{d_{p}^{N},d_{q}^{N}})/\!\!\sim $ be the equivalent class of $N$-angulations. Then there is a bijective map\begin{equation*}
		\widetilde{\sigma}_{d_p^N,d_q^N}: \tilde{\mathcal{G}}(P_{d_{p}^{N},d_{q}^{N}})/\!\!\sim \,\rightarrow \mathcal{M}_{d_p^N,d_q^N}.
	\end{equation*}
\end{thm}

\subsection{Cluster categories and exchange graphs}

Let $Q$ be an acyclic quiver and $D^b(kQ)$ be the bounded derived category of the finite dimensional left module over the path algebra of $Q$.
The Auslander-Reiten translation $\tau$ is an autoequivalence on $D^b(kQ)$ such that $[1]\tau$ is a Serre functor.
The $(N-2)$-cluster category $\mathcal{C}_{N-2}(Q)$ is 
defined to be the orbit category 
\begin{equation*}
	\mathcal{C}_{N-2}(Q):=D^b(kQ) \slash \left<\tau^{-1} \circ [N-2] \right>.
\end{equation*}

Assume that $Q$ has $n$ vertices. 
An $(N-2)$-cluster tilting set $\{Y_i\}_i^n$ in $\mathcal{C}_{N-2}(Q)$ is a maximal collection of non-isomorphic indecomposables with $Ext_{\mathcal{C}_{N-2}(Q)}^{k}(Y_i,Y_j)=0$ for all $1\leq k\leq N-2$.
The sum $Y=\bigoplus_{i=1}^{n}Y_i$ is called an ($N-2$)-cluster tilting object. One can construct a new 
($N-2$)-cluster tilting object $\mu_i^{\sharp} Y$, called the forward  mutation \cite{KQ} of $Y$ at the $i$-th object $Y_i$ defined by replacing $Y_i$
\begin{equation*}
	Y_i^{\sharp}=\mathrm{Cone}(Y_i\rightarrow\bigoplus_{j\neq i}\mathrm{Irr}(Y_i,Y_j)^*\otimes Y_j)
\end{equation*}
where $\mathrm{Irr}(Y_i,Y_j)$ is a space of irreducible maps $Y_i\rightarrow Y_j$ in the additive subcategory of $\mathcal{C}_{N-2}(Q)$. The cluster exchange grpah $\mathrm{CEG}_{N-2}(Q)$ is an oriented graph whose vertices are $(N-2)$-cluster objects and whose labeled edges are forward mutations. For each ($N-2$)-cluster tilting object, we can define an ($N-2$)-colored quiver $Q(Y)$ \cite[section 4]{KQ}.

We set $D(Q_N):=D_{fd}(k\Gamma_NQ)$ the derived category of dg modules over $k\Gamma_NQ$ with finite total dimension. Define the exchange graph $\EG(D(Q_N))$ of $D(Q_N)$ whose vertices are finite hearts of $D(Q_N)$, and edges are forward tiltings between hearts. The connected subgraph $\EG^{\circ}(D(Q_N)) \subset \EG(D(Q_N))$ is 
defined to be the connected component which contains the standard heart.
King and Qiu established the relation between and cluster exchange graph $\mathrm{CEG}_{N-2}(Q)$ and the exchange graph $\EG(D(Q_N))$ (\cite{KQ}, Theorem 8.6). When $Q$ is a quiver of affine type $A$, we have the following.

\begin{thm}[\cite{KQ}, Theorem 8.6]\label{KG}
	There is an isomorphism of oriented graphs with labeled edges
	\begin{equation*}
		\EG^{\circ}(D(\tilde{A}_{n,N})) / \mathrm{Sph}(\tilde{A}_{n,N})
		\cong \mathrm{CEG}_{N-2}(\tilde{A}_{n,N}).
	\end{equation*}
\end{thm}

Let $\mathcal{H}\in\EG^{\circ}(D(\tilde{A}_{n,N}))$ be a finite heart and $Y\in\mathrm{CEG}_{N-2}(\tilde{A}_{n,N})$
be the corresponding ($N-2$)-cluster tilting object by the above theorem. Then we have 
\begin{equation}\label{QHY}
	Q(\mathcal{H})=Q(Y).
\end{equation}

The mutation class of colored quivers are in bijection to mutation class of $(N-2)$-cluster tilting objects.
Theorem \ref{equiv} can be restated as the following:
\begin{cor}\label{T}
	There is an isomorphism of oriented graphs with labeled arrows 
	\begin{equation*}
		\tilde{\mathcal{G}}(P_{d_{p}^{N},d_{q}^{N}}) \xto{\sim} \mathrm{CEG}_{N-2}(\tilde{A}_{n,N}).
	\end{equation*}  
\end{cor}

Combining corollary \ref{T} and theorem \ref{KG}, we have the geometric description of the heart exchange graph for a quiver of affine type $A$.
\begin{cor}\label{main}
	There is an isomorphism of oriented graphs with labeled edges between the set of $N$-angulation of annulus and the corresponding cluster exchange graph. we obtain an isomorphism of oriented graphs with labeled arrows:
	\begin{equation*}
		\tilde{\mathcal{G}}(P_{d_{p}^{N},d_{q}^{N}}) \xto{\sim} \EG^{\circ}(D(\tilde{A}_{n,N}))  /\mathrm{Sph}(\tilde{A}_{n,N}).
	\end{equation*}
	For $\Delta\in\tilde{\mathcal{A}}(P_{d_{p}^{N},d_{q}^{N}})$ and the corresponding heart $H(\Delta)\in \EG^{\circ}(D(\tilde{A}_{n,N}))$, there is a canonical bijection between $(N-2)$-diagonals $\{\delta_1,\cdots\delta_n\}$ and simple objects $\{S_1,\cdots S_n\}$ by sending $\delta_i$ to $S_i$.
\end{cor}
Corollary \ref{main} implies the following result.
\begin{lem}
	\label{HS}
	Let $\Delta\in\tilde{\mathcal{A}}(P_{d_{p}^{N},d_{q}^{N}})$ be an $N$-angulation and $\mathcal{H}(\Delta) \in \EG_N^{\circ}(D(\tilde{A}_{n,N}))$ 
	be the corresponding heart. 
	Take a $(N-2)$-diagonal $\delta_v \in \Delta$ and consider 
	the new $N$-angulation $\mu_{\delta_v}^{\sharp}\Delta$. 
	Then there is a unique spherical twist $\Phi \in \mathrm{Sph}(\tilde{A}_{n,N}) $ such that 
	\begin{equation*}
		\Phi \cdot \mathcal{H}(\mu_{\delta_v}^{\sharp}\Delta)=
		\mathcal{H}(\Delta)_{S_v}^{\sharp}.
	\end{equation*}
\end{lem}
\begin{proof}
	By corollary \ref{main}, such $\Phi$ exists. By \cite[proposition 4.17]{Ike16}, the action of $\mathrm{Sph}(\tilde{A}_{n,N})$ on $\EG_N^{\circ}(D(\tilde{A}_{n,N}))$ is free.
\end{proof}	

By the construction in section 4.1 and corollary \ref{T}, the colored quiver $Q_{\Delta}$ is equal to the colored quiver $Q(Y_{\Delta})$ associated to the cluster tilting object $Y_{\Delta}$. The formula (\ref{QHY}) implies that the colored quiver $Q(Y_{\Delta})$ is the same as $Q(\mathcal{H}(\Delta))$. Therefore there is an isomorphism of colored quivers
\begin{equation}\label{QH}
	Q_{\Delta}=Q(\mathcal{H}(\Delta)).
\end{equation}

\subsection{Colored quiver lattices}
\label{sec_mutation_K}

We define a lattice associated to a 
colored quiver \cite{Ike16}. For an $(N-2)$-colored quiver  $Q$ with vertices $\{1,\dots,n\}$, 
we introduce a free abelian group $L_{Q}$ 
generated by vertices of $Q$. 
Denote by $\alpha_1,\dots,\alpha_n$ generators corresponding to vertices $1,\dots,n$. 
Then we have 
\begin{equation*}
	L_{Q} = \bigoplus_{i=1}^n \mathbb{Z} \,\alpha_i.
\end{equation*} 
Define the bilinear form $\left<\,,\,\right> \colon Q \times Q \lto \mathbb{Z}$ by 
\begin{equation*}
	\left<\alpha_i,\alpha_j \right> := \delta_{ij} + (-1)^{N}\delta_{ij} -\sum_{c=0}^{N-2} (-1)^c q_{ij}^{(c)}.
\end{equation*}
Since $q_{ij}^{(c)}= q_{ji}^{(N-2-c)}$, the bilinear form 
$\left<\,,\,\right>$ is symmetric if $N$ is even, 
and skew-symmetric if $N$ is odd.

Let $Q(\mathcal{H})$ be the colored quiver 
of a heart $\mathcal{H} \in \EG^{\circ}(D(Q_N))$ with 
simple objects $S_1,\dots,S_n$ and denote by $\alpha_i$ the basis of 
colored quiver lattice $L_{Q(\mathcal{H})}$ 
which corresponds to a vertex $i$ of $Q(\mathcal{H})$. 
Then there is an isomorphism of abelian groups \cite[Lemma 4.27]{Ike16},
\begin{equation}\label{LK}
	\lambda_{\mathcal{H}} \colon L_{Q(\mathcal{H})} \lto K(\mathcal{H}) \cong K(D(Q_N))
\end{equation}
defined by $\alpha_i \mapsto [S_i]$. The map $\lambda$ 
takes the bilinear form $\left<\,,\,\right>$ on $L_{Q(\mathcal{H})}$ 
to the Euler form $\chi$ on $K(D(Q_N))$.

We write by $\Delta_{\phi}$ the $N$-angulation 
determined by a saddle-free differential $\phi \in \mathrm{Quad}(n,N)$. 
We also write by $h(\delta)$ the horizontal strip which contains a generic trajectory $\delta$.  We can associate standard saddle classes $\gamma_1,\dots,
\gamma_n \in H_{\pm}(\phi)$ to horizontal strips $h(\delta_1),
\dots, h(\delta_ n)$. Set $\Delta_{\phi}=\{\delta_1,\dots,\delta_n\}$ 
and consider the colored quiver lattice 
\begin{equation*}
	L_{Q(\Delta_{\phi})}=\oplus_{i=1}^n \mathbb{Z} \alpha_i
\end{equation*}
where $\alpha_1,\dots,\alpha_n$ are the basis 
corresponding to $\delta_1,\dots,\delta_n$.

Thus, we can establish an isomorphism between the period map of the quadratic differential and the Grothendieck group on the derived category $D(\tilde{A}_{n,N})$.
\begin{cor} 
	\label{cor_PC}
	There is an isomorphism of 
	abelian groups
	\begin{equation}
		\nu_{\phi} \colon K(D(\tilde{A}_{n,N})) \cong K(\mathcal{H}(\Delta_{\phi}))\to H_{\pm}(\phi)
	\end{equation}  
	defined by $[S_i] \mapsto \gamma_i$. 
	The map $\nu_{\phi}$ takes  
	the Euler form $\chi$ on $K(D(\tilde{A}_{n,N}))$ to the intersection form $I_{\pm}$ on $H_{\pm}(\phi)$. 
\end{cor}

\begin{proof}
	
	The argument of \cite[Corollary 6.13]{Ike16} for the quiver of type $A$ can be applied to the case of the  affine type $A$ quiver.
	There is an isomorphism of abelian groups
	\begin{equation}\label{LH}
		\mu_{\phi} \colon L_{Q(\Delta_{\phi})}
		\to  H_{\pm}(\phi)
	\end{equation}
	defined by $\alpha_i \mapsto \gamma_i$. The map $\mu_{\phi}$
	takes the bilinear form $\left<\,,\,\right>$ on 
	$L_{\mathcal{Q}(\Delta_{\phi})}$ to 
	the intersection form $I_{\pm}$ on $H_{\pm}(\phi)$. By (\ref{QH}), 
	there is an isomorphism $L_{Q(\Delta_{\phi})}\simeq L_{Q(\mathcal{H}(\Delta_{\phi}))}$. Combining (\ref{LK}), we have the result. 
\end{proof}

\section{Bridgeland Stability Conditions}
Bridgeland introduced the notion of stability conditions on triangulated categories \cite{Br1}. We recall some results for the space of stability conditions.
\begin{defi}
	Let $\D$ be a triangulated category and $K(\D)$ be its $K$-group. 
	A stability condition $\sigma = (Z, \sli)$ on $\D$ which consists of 
	a group homomorphism $Z \colon K(\D) \to \mathbb{C}$ called central charge and 
	a family of full additive subcategories $\sli (\theta) \subset \D$ for $\theta \in \mathbb{R}$ 
	satisfying the following conditions:
	\begin{itemize}
		\item[(a)]
		if  $0 \neq E \in \sli(\theta)$, 
		then $Z(E) = m(E) \exp(i \pi \theta)$ for some $m(E) \in \mathbb{R}_{>0}$,
		\item[(b)]
		for all $\theta \in \mathbb{R}$, $\sli(\theta + 1) = \sli(\theta)[1]$, 
		\item[(c)]if $\theta_1 > \theta_2$ and $A_i \in \sli(\theta_i)\,(i =1,2)$, 
		then $\Hom_{\D}(A_1,A_2) = 0$,
		\item[(d)]for $0 \neq E \in \D$, there is a finite sequence of real numbers 
		\begin{equation*}
			\theta_1 > \theta_2 > \cdots > \theta_m
		\end{equation*}
		and a collection of exact triangles
		\begin{equation*}
			0 =
			\xymatrix @C=5mm{ 
				E_0 \ar[rr]   &&  E_1 \ar[dl] \ar[rr] && E_2 \ar[dl] 
				\ar[r] & \dots  \ar[r] & E_{m-1} \ar[rr] && E_m \ar[dl] \\
				& A_1 \ar@{-->}[ul] && A_2 \ar@{-->}[ul] &&&& A_m \ar@{-->}[ul] 
			}
			= E
		\end{equation*}
		with $A_i \in \sli(\theta_i)$ for all $i$.
	\end{itemize}
\end{defi}

It follows from the definition that the subcategory 
$\sli(\theta) \subset \D$ is an abelian category 
(see Lemma 5.2 in \cite{Br1}). 
Nonzero objects in $\sli(\theta)$ are called semistable of phase $\theta$ in $\sigma$ and simple objects 
in $\sli(\theta)$ are called stable of phase $\theta$ in $\sigma$. 

We also assume the stability conditions $\sigma=(Z,\mathcal{P})$ satisfy the support property:
For some norm $\parallel\cdot\parallel$ on $k(\D)\otimes\mathbb{R}$  there is a constant $C>0$ such that \begin{equation*}
	\parallel\gamma\parallel<C\cdot|Z(\gamma)|
\end{equation*}
for all classes $\gamma\in K(\D)$ of semistable objects in $\D$. Bridgeland also introduced a topology on the stability space with support property and showed that it is a complex manifold.
\begin{thm}[\cite{Br1}, Theorem 1.2]
	\label{localiso}
	The space $\Stab(\D)$ has the structure of a complex manifold and 
	the projection map of central charges
	\begin{equation*}
		\mathcal{Z} \colon \Stab(\D) \longrightarrow \Hom_{\mathbb{Z}}(K(\D),\mathbb{C})
	\end{equation*}
	defined by $(Z,\sli) \mapsto Z$ is a local isomorphism of complex manifolds 
	onto an open subset of $\Hom_{\mathbb{Z}}(K(\D),\mathbb{C})$. 
\end{thm}

We denote by $\Stab(D(\tilde{A}_{n,N}))$ the space of stability conditions of the derived category of right modules over the dg Ginzburg algebra of the quiver of affine type $A$ with the support property. We denote by $\Stab^{\circ}(D(\tilde{A}_{n,N}))$ the distinguished connected component which contains the connected subset containing the standard heart. 

By composing an isomorphism $\nu_{\phi}:K(\mathcal{H}(\Delta_{\phi}))\rightarrow H_{\pm}(\phi)$ and a period of $\phi$, $Z_{\phi}:H_{\pm}(\phi)\rightarrow\mathbb{C}$.  We construct a central charge from a saddle-free differential $Z:=Z_{\phi}\circ \nu_{\phi}:K(\mathcal{H}(\Delta_{\phi}))\rightarrow \mathbb{C}$.

\subsection{Wall-crossing}
We describe the wall-crossing in the moduli space of quadratic differentials. By proposition 5.5 and 5.7 of \cite{BrSm}, we can deform a differential $\phi_1$ which has exactly one saddle trajectory by rotating a sufficiently small phase $t$ to obtain two saddle free differentials:
\begin{equation*}
	\phi=e^{+i\pi t}\phi_1,\quad\phi'=e^{-i\pi t}\phi_1
\end{equation*}  
Generic trajectories in the horizontal decomposition of saddle free differentials $\phi$ and $\phi'$ would determine $N$-angulations $\Delta_{\phi}$ and $\Delta_{\phi'}$ respectively. There is a unique $N-2$ diagonal $\delta_v$ of $\Delta_{\phi}=\{\delta_1\cdots\delta_n\}$ which does not belong to the set of $N-2$ diagonals by taking a generic trajectory in each horizontal strip of $\phi'$. By definition of forward mutated $N$-angulation, two $N$-angulation are related by $\mu^{\sharp}_{\delta_v}(\Delta_{\phi})=\Delta_{\phi'}$ \cite[Lemma 7.5]{Ike16}. 

The corresponding colored quivers $Q(\Delta_{\phi})$ mutated at the vertex $v$ corresponding to the $N-2$ diagonal would be  
\begin{equation*}
	\mu_v^{\sharp}Q(\Delta_{\phi})=Q(\mu_{\delta_v}^{\sharp}\Delta_{\phi})=Q(\Delta_{\phi'}),
\end{equation*}
and there is an isomorphism of colored quiver lattice $L_{\mu_v^{\sharp}Q(\Delta_{\phi})}=L_{Q(\Delta_{\phi'})}$.

For $\phi$ and $\phi'$, we take basis $\{\gamma_1,\cdots\gamma_n\}$ and $\{\gamma'_1,\cdots\gamma'_n\}$ for homology classes $H_{\pm}(\phi)$ and $H_{\pm}(\phi')$,
\begin{equation*}
	H_{\pm}(\phi)=\bigoplus_{i=1}^{n}\mathbb{Z}\gamma_i,\quad H_{\pm}(\phi')=\bigoplus_{i=1}^{n}\mathbb{Z}\gamma_i'
\end{equation*}
and the framing isomorphisms $\psi: \Gamma\rightarrow H_{\pm}(\phi) $ and $\psi': \Gamma\rightarrow H_{\pm}(\phi') $. We connect $(\phi,\psi)$ and $(\phi',\psi')$ by a continuous path in a connected component $\mathrm{Quad}(n,N)^{\Gamma}_*$ containing $(\phi,\psi)$. By proposition 5.5 and 5.7 in \cite{BrSm}, this path can be deformed to lie in the subset which contains at most one saddle trajectory  and intersects transversally to the locally closed subset which contains just one saddle trajectory.

Denote by $\{\alpha_1,\cdots\alpha_n\}$ and $\{\alpha'_1,\cdots\alpha'_n\}$  the collection of generators corresponding to vertices $1,\cdots,n$ for the colored quiver lattices $L_{Q(\Delta_{\phi})}$ and $L_{Q(\Delta_{\phi'})}$:
\begin{equation*}
	L_{Q(\Delta_{\phi})}=\bigoplus_{i=1}^{n}\mathbb{Z}\alpha_i,\quad L_{Q(\Delta_{\phi'})}=\bigoplus_{i=1}^{n}\mathbb{Z}\alpha_i'.
\end{equation*}
Define a linear map $F_v:L_{Q(\Delta_{\phi'})}\simeq L_{\mu_v^{\sharp}Q(\Delta_{\phi})}\longrightarrow L_{Q(\Delta_{\phi})}$ by

\begin{equation*}
	F_v(\alpha'_j):=\alpha_i+q^{(0)}_{vj}\alpha_v.
\end{equation*}

We denote by $\mu_{\phi} \colon L_{Q(\Delta_{\phi})} \to H_{\pm}(\phi)$ and
$\mu_{\phi'} \colon L_{Q(\Delta_{\phi'})} \to H_{\pm}(\phi')$ the isomorphisms given by (\ref{LH}) \cite[Lemma 6.12]{Ike16}. 
Thus, we have a commutative diagram:
\begin{equation}\label{56}
	\xymatrix{ 
		H_{\pm}(\phi') \ar[rr]^{\mathrm{GM}_c} && H_{\pm}(\phi)
		&\\
		L_{Q(\Delta_{\phi'})}\ar[u]^{\mu_{\phi'}} \ar[rr]^{\mathrm{F}_v} && L_{Q(\Delta_{\phi})}\ar[u]_{\mu_{\phi}} 
	}
\end{equation}
By lemma \ref{HS} and the relation $\mu^{\sharp}_{\delta_v}(\Delta_{\phi})=\Delta_{\phi'}$, there exists a spherical twist $\Phi$ satisfying $\Phi\cdot \mathcal{H}(\Delta_{\phi'})=\mathcal{H}(\Delta_{\phi})_{S_v}'$. There are also isomorphisms of lattice $\lambda:L_{Q(\Delta)}\rightarrow K(D(\tilde{A}_{n,N}))$ given by $\lambda(\alpha)=[S_i]$ and $\lambda':L_{Q(\Delta_{\phi'})}\rightarrow K(D(\tilde{A}_{n,N}))$ given by $\lambda'(\alpha_i')=[S'_i]$. By lemma 6.10 of \cite{Ike16}, the diagram of colored quiver lattices and $K$ groups of $D(A_{n,N})$ under the wall-crossing:
\begin{equation}\label{27}
	\xymatrix{ 
		L_{Q(\Delta_{\phi'})}\ar[d]_{\lambda_{\phi'}} \ar[rr]^{\mathrm{F}_v}  && L_{Q(\Delta_{\phi})}\ar[d]^{\lambda_{\phi}} 
		&\\
		K(D(\tilde{A}_{n,N}))\ar[rr]^{[\Phi]} &&  K(D(\tilde{A}_{n,N}))
	}
\end{equation}
commutes. Let $\nu_{\phi}:K(D(\tilde{A}_{n,N}))\rightarrow H_{\pm}(\phi)$ and $\nu_{\phi'}:K(D(\tilde{A}_{n,N}))\rightarrow H_{\pm}(\phi')$ be isomrphisms given by $\nu([S_i])=\gamma_i$ and $\nu([S'_i])=\gamma'_i$ as in the corollary \ref{cor_PC}.
Combining diagrams (\ref{56}) and (\ref{27}), we have the commutative diagram.
\begin{equation}\label{30}
	\xymatrix{     
		&K(D(\tilde{A}_{n,N})) \ar[d]_{\nu_{\phi}} \ar[rr]^{[\Phi]} && K(D(\tilde{A}_{n,N}))\ar[d]^{\nu_{\phi'}} &\\
		&H_{\pm}(\phi) \ar[rr]^{\mathrm{GM}_c} && H_{\pm}(\phi')
		&\\
	}
\end{equation}
Let $(\phi,\psi)$ and $(\phi',\psi')$ be framed quadratic differentials, where $\phi$ and $\phi'$ are as above.
We define the isomorphisms $\kappa_{\phi}$ and $\kappa_{\phi}'$ by composition of $\nu_{\phi}^{-1}$ and $\nu_{\phi'}^{-1}$ with the framing isomorphisms $\psi$ and $\psi'$:
\begin{align}\label{GK}
	&\kappa_{\phi}:=\nu_{\phi}^{-1}\circ\psi:\Gamma\rightarrow K(\mathcal{H}(\Delta_{\phi}))\simeq K(D(\tilde{A}_{n,N})),\\
	&\kappa_{\phi'}:=\nu_{\phi'}^{-1}\circ\psi':\Gamma\rightarrow K(\mathcal{H}(\Delta_{\phi'}))\simeq K(D(\tilde{A}_{n,N})).
\end{align}
Let $(\phi_0,\psi_0)\in\mathrm{Quad}(N,n)^{\Gamma}$ be a fixed framed differential and denote by $(\phi,\psi)\in\mathrm{Quad}(N,n)^{\Gamma}_*$ the connected component containing $(\phi_0,\psi_0)$.
\begin{lem}\label{tw}
	Let $(\phi,\psi)\in\mathrm{Quad}(N,n)^{\Gamma}_*$ be a frame saddle-free differentials and $\kappa_{\phi}$ and $\kappa_{\phi'}$ be isomorphisms defined by (\ref{GK}). The there is a spherical twist $\Phi\in\mathrm{Sph}(D(\tilde{A}_{n.N}))$ such that the diagram
	\begin{equation}
		\xymatrix{ 
			& \Gamma \ar[dl]_{\kappa_{\phi}} \ar[dr]^{\kappa_{\phi'}}  & \\ 
			K(D(\tilde{A}_{n,N})) \ar[rr]^{[\Phi]} && K(D(\tilde{A}_{n,N})) 
		}
	\end{equation}
	commutes.
\end{lem}
\begin{proof}
	Combining the diagram (\ref{30}) and the commutative diagram (\ref{gm}), we can see that there is a spherical twist $\Phi\in \mathrm{Sph}(D(\tilde{A}_{n.N}))$ that makes the diagram commutes.
	
\end{proof}	
If there are two such spherical twists $\Phi_1$ and $\Phi_2$, their induced isomorphisms on $K(D(\tilde{A}_{n,N}))$ are the same since
\begin{equation*}
	[\Phi_1]\circ \kappa_{\phi}=\kappa_{\phi'}=[\Phi_2]\circ \kappa_{\phi}.
\end{equation*}

\subsection{Main theorem}
In this subsection, we state the main results of this paper. Donote by $s_{\phi}$ the number of saddle trjectories
Define the subset of $\mathrm{Quad}(N,n)$
\begin{equation*}
	D_p:=\{\phi\in \mathrm{Quad}(N,n)\,|\,s_{\phi}\leq p\,\}
\end{equation*}
There is a stratification of $\mathrm{Quad}(N,n)$
\begin{equation*}
	D_0\subset D_1\subset\cdots D_n=\mathrm{Quad}(N,n).
\end{equation*}
This stratification can also be extended to $\mathrm{Quad}(N,n)^{\Gamma}_*$ by the obvious way.
We can construct a stability condition from a saddle free differential $\phi\in\mathrm{Quad}(N,n)$.
\begin{lem}\cite[Lemma 7.4]{Ike16}
	Let $\phi\in D_0\subset\mathrm{Quad}(N,n)$ be a saddle free diffewrential. Consider the correponding hear $\mathcal{H}(\Delta_{\phi})\in\EG_N^{\circ}(\mathcal{H})$ and define a linear map
	\begin{equation*}
		Z:K(\mathcal{H}(\Delta_{\phi}))\rightarrow \mathbb{C}
	\end{equation*}
	by $Z:=Z_{\phi}\circ\nu_{\phi}$ where $Z_{\phi}:H_{\pm}(\phi)\rightarrow\mathbb{C}$ is aperiod map of $\phi$ and $\nu_{\phi}:K(\mathcal{H}(\Delta_{\phi}))\rightarrow H_{\pm}(\phi)$ is an isomorphism given by corollary \ref{cor_PC}. Then the pair $(Z,\mathcal{H}(\Delta_{\phi}))$ determine a stability condition $\sigma(\phi)=(Z,\mathcal{H}(\Delta_{\phi}))\in\mathrm{Stab}^{\circ}(D(\tilde{A}_{n,N}))$.
\end{lem}

Let $\mathrm{Sph}^{0}(\tilde{A}_{n,N})$ be the subgroup which acts as the identity on $K(D(\tilde{A}_{n,N}))$.

\begin{prop}
	There is a $\mathbb{C}$-equivariant holomorphic map $K_N$ such that the diagram 
	\begin{equation}\label{66}
		\xymatrix{     
			\mathrm{Quad}(N,n)^{\Gamma}_* \ar[rr]^{K_N} \ar[dr]_{\mathcal{W}_N} && \mathrm{Stab}^{\circ}(D(\tilde{A}_{n,N}))/\mathrm{Sph}^{0}(\tilde{A}_{n,N})\ar[dl]^{\mathcal{Z}} \\
			& \mathrm{Hom}_{\mathbb{Z}}(\Gamma,\mathbb{C})&\\
		}
	\end{equation}
	commutes. Where $\mathcal{W}_N$ and $\mathcal{Z}$ are local isomorphisms in (\ref{per}) and theorem \ref{localiso}.
\end{prop}

\begin{proof}
	
	The proof of this proposition is the same in \cite{Ike16} and \cite{BrSm}, so we only sketch the idea.	
	We construct the map $K_N$ from the moduli space of quadratic differentials to the space of stability conditions on the stratum $D_0$. This map $K_N$ can be extend on other stratum $D_p$ by \cite[propsition 11.3]{BrSm}. 
	Let $\Phi$ be a spherical twist 
	
	For a framed saddle free quadratic differential $(\phi,\psi)\in \mathrm{Quad}(N,n)^{\Gamma}_*$ and a stability condition $\sigma(\phi)\in\mathrm{Stab}^{\circ}(D(\tilde{A}_{n,N}))$, we define the map
	$K_N$ by	
	\begin{equation*}
		K_N((\phi,\psi)):=\Phi(\sigma(\phi))\in\mathrm{Stab}^{\circ}(D(\tilde{A}_{n,N}))/\mathrm{Sph}^{0}(\tilde{A}_{n,N}),
	\end{equation*}	
	where $\Phi$ is given by lemma \ref{tw}. 
	
	The action on the $\mathrm{Quad}(N,n)^{\Gamma}_*$ is the pullback of the standard $\mathbb{C}^*$ action rescaling the quadratic differential. For $t\in \mathbb{C}$ and $(Z,\mathcal{P})$, there is a $\mathbb{C}$-action on $\mathrm{Stab}^{\circ}(D(\tilde{A}_{n,N}))$ defined by 
	\begin{equation*}
		Z'(E):=e^{-i\pi t}\cdot Z(E),\quad \mathcal{P}'(\theta)=\mathcal{P}(\theta+\mathrm{Re}(t)),
	\end{equation*}
	where $E\in D(\tilde{A}_{n,N})$ and $\theta\in\mathbb{R}$. The map $\mathcal{Z}$ is also $\mathbb{C}$-equivariant.
	Since $\mathcal{Z}$ and $\mathcal{W}_N$ are local isomorphisms by Theorem \ref{localiso} and \cite[Theorem 4.12]{BrSm}, we can conclude that $K_N$ is $\mathbb{C}$-equivariant.
\end{proof}	

Following the same argument in \cite[section 11]{BrSm}, we can construct the isomorphism of the complex manifold:
\begin{thm}
	\label{map}
	There is a $\mathbb{C}$-equivariant 
	holomorphic map of complex orbifolds  
	\begin{equation*}
		\mathrm{Quad}(N,n)_*^{\Gamma} \rightarrow \Stab^{\circ}(D(\tilde{A}_{n,N})) / \mathrm{Sph}^0(\tilde{A}_{n,N}) \\
	\end{equation*}
	In particular, this map is an isomorphism.
\end{thm}

The moduli space $\mathrm{Quad}(N,n)$ of the quadratic differential $\phi(z)$ is isomorohic to the configuration space $\mathrm{Conf}^{n}(\mathbb{C}^{*})$ of $n$ distinct points in $\mathbb{C}^{*}$ quotient by $(\mathbb{Z}/d_{p}^{N})$ which acts by multiplication by a $d_{p}^{N}$-th root of unity. In the case $p=n/2$ we need to quotient by an extra factor of $\mathbb{Z}_2$ acting by $z\rightarrow1/z$.

Assume that $n \ge 3$. Then 
there is a short exact sequence of groups \cite[Theorem 9.10]{BrSm}
\begin{equation*}
	1 \to \mathrm{Sph}(\tilde{A}_{n,N}) \to \mathrm{Aut}(D(\tilde{A}_{n,N})
	\to \mathrm{MCG}(P_{d_{p}^{N},d_{q}^{N}}) \to 1.
\end{equation*}
From theorem \ref{faithful}, there is an isomorphism
\begin{equation*}
	B_{\tilde{A}_{n}}\iso \mathrm{Sph}(\tilde{A}_{n,N}), \quad \sigma_i \mapsto \mathrm{Tw}_{S_i}.
\end{equation*}
If $p\neq q$ the mapping class group of the marked annulus $P_{d^N_p,d^N_q}$ is given by \cite{ARS} 
\begin{equation*}
	H_{d_p^N,d_q^N}:=\{\, (r_1,r_2)\,|\,(r_1)^{d_p^N}=(r_2)^{d_q^N},\, r_1r_2=r_2r_1\,\}.
\end{equation*}
If $p=q$, $p\geq 2$ there is an extra symmetry interchanging the interior boundary and the exterior boundary so the mapping class group in this case is $H_{d_p^N,d_p^N}\rtimes\mathbb{Z}\slash 2$. Theofore, we can determine the autoequivalence group of the derived category of affine type $A$ quiver algebras.

\section*{Acknowledgements}

This article is part of the author's PhD thesis under Wu-yen Chuang's supervision. I am grateful for his constant encouragement and support.
I would like to thank Mohammed Abouzaid, Mandy-Cheung, Yu-Wei Fan, Joshua Sussan for useful discussion. I also would like to thank Akishi Ikeda and Hermund Torkildsen for the email correspondence and the anonymous referee for many suggestions. Part of this paper was done while I visited Columbia university, I appreciate their warm hospitality and providing me wonderful research environment. This work was supported by ministry of science and technology, Taiwan.


\newpage
	\begin{thebibliography}{BBD82}
	
	
	
	
	\bibitem{ACC}		
	Alim, M., Cecotti, S., Córdova, C. et al. 
	\newblock
	BPS Quivers and Spectra of Complete N=2 Quantum Field Theories. 
	\newblock
	Commun. Math. Phys. 323, 1185–1227 (2013).		
	
	
	
	\bibitem{Am}
	C~Amiot.
	\newblock Cluster categories for algebras of global dimension 2 and quivers
	with potential.
	\newblock {\em Ann. Inst. Fourier (Grenoble)}, 59(6):2525--2590, 2009.
	
	
	
	\bibitem{Br1}
	T.~Bridgeland.
	\newblock Stability conditions on triangulated categories.
	\newblock {\em Ann. of Math.}, 166(2):317--345, 2007.
	
	\bibitem{Br4}
	T.~Bridgeland.
	\newblock Spaces of stability conditions.
	\newblock In {\em Algebraic geometry---{S}eattle 2005. {P}art 1}, volume~80 of
	{\em Proc. Sympos. Pure Math.}, pages 1--21. Amer. Math. Soc., Providence,
	RI, 2009.
	
	
	
	\bibitem{BrSm}
	T.~Bridgeland and I.~Smith.
	\newblock Quadratic differentials as stability conditions.
	\newblock {\em Publ. Math. de l'IHES.}, 121:155--278, 2015.
	
	\bibitem{BT}
	A.~B. Buan and H.~Thomas.
	\newblock Coloured quiver mutation for higher cluster categories.
	\newblock {\em Adv. Math.}, 222(3):971--995, 2009.
	
	
	\bibitem{ARS} 
    I.~ Assem, R.~ Schiffler and V.~ Shramchenko. 
    \newblock Cluster automorphisms. 
    \newblock Proceedings of the London Mathematical Society, 104(6):1271–1302, 2012.
	
	
	\bibitem{DFR}
	M.~Douglas, B.~Fiol and Christian Römelsberger
	\newblock Stability and BPS branes.
	\newblock Journal of High Energy Physics, Volume 2005(09); p. 006
	
	\bibitem{FM}
	S.~Franco and G.~Musiker
	\newblock
	Higher Cluster Categories and QFT Dualities.
	\newblock
	Phys. Rev. D98, 046021 (2018)
	
	
	\bibitem{GTW}
	A.~Gadbled and A.~Thiel and E.~Wagner
	\newblock Categorical action of the extended braid group of affine type A.
	\newblock {\em  Commun. Contemp. Math.}, 19, 1650024 (2017) 
	
	
	\bibitem{Ginz}
	V.~G. Ginzburg.
	\newblock Calabi-{Y}au algebras.
	\newblock arXiv:math/0612139.
	
	
	\bibitem{Ike16}
	A.~Ikeda.
	\newblock Stability conditions on $CY_N$ categories associated to $A_n$-quivers and period maps.
	\newblock {\em Mathematische Annalen}, (2017) 367: 1. 
	
	\bibitem{IQ}
	A.~Ikeda and Y.~Qiu,
	\newblock
	q-Stability conditions via q-quadratic differentials for Calabi-Yau-X categories.
	\newblock  arXiv:1812.00010 
	
	\bibitem{KKV}
	S.~ Katz, A.~ Klemm and C.~ Vafa
	\newblock
	Geometric engineering of quantum field theories
	\newblock Nuclear Physics B
	Volume 497, Issues 1–2, 21 July 1997, Pages 173-195
	
	
	
	
	\bibitem{Kel2}
	B~Keller.
	\newblock Deformed {C}alabi-{Y}au completions.
	\newblock {\em J. Reine Angew. Math.}, 654:125--180, 2011.
	\newblock With an appendix by Michel Van den Bergh.
	
	\bibitem{KN}
	B.~Keller and P.~Nicol\'as.
	\newblock Weight structures and simple dg modules for positive dg algebras
	\newblock {\em Int. Math. Res. Not.}, 1028--1078, 2013.
	\bibitem{KQ}
	K.~King and Y.~Qiu.
	\newblock Exchange graphs and {E}xt quivers.
	\newblock to appear in {\em Adv. Math.}, arXiv:1109.2924v3.  
	
	
	\bibitem{KhSe}
	M.~Khovanov and P.~Seidel.
	\newblock Quivers, floer cohomology, and braid group actions.
	\newblock {\em J. Amer. Math. Soc.}, 15:203--271, 2002.
	
	
	
	\bibitem{KY}
	B.~Keller and D.~Yang.
	\newblock Derived equivalences from mutations of quivers with potential.
	\newblock {\em Adv. Math.}, 226(3):2118--2168, 2011.
	
	\bibitem{Qiu}
	Y.~Qiu.
	\newblock Decorated marked surfaces: spherical twists versus braid twists.
	\newblock {\em Math. Ann.}, 365(2016), pp 595-633.
	
	\bibitem{ST}
	P.~Seidel and R.~P. Thomas.
	\newblock Braid group actions on derived categories of coherent sheaves.
	\newblock {\em Duke Math. J.}, 108(1):37--108, 2001.
	
	\bibitem{Str}
	K~Strebel.
	\newblock Quadratic differentials. volume~5 of {\em Ergebnisse der
		Mathematik und ihrer Grenzgebiete (3)}.
	\newblock Springer-Verlag, Berlin, 1984.
	
	\bibitem{SW}
	N,~Seiberg and E,~Witten
	\newblock
	Electric-magnetic duality, monopole condensation, and confinement in N=2 supersymmetric Yang-Mills theory
	\newblock
	Nucl.Phys.B 426 (1994) 19-52, Nucl.Phys.B 430 (1994) 485-486 (erratum)
	
	\bibitem{Tor12}
	H.~A.~Torkildsen.
	\newblock A geometric realization of {$m$}-cluster categories of type $\tilde{A}$.
	\newblock {\em Trans. Amer. Math. Soc.}, 360(11):5789--5803, 2008.
	
	
\end{thebibliography}
\end{document}